\documentclass[12pt]{amsart}
\usepackage{amscd,amssymb,verbatim,bbm,comment}
\newtheorem{thm}{Theorem}[section]
\newtheorem{cor}[thm]{Corollary}
\newtheorem{lem}[thm]{Lemma}
\newtheorem{prop}[thm]{Proposition}

\newtheorem{rem}[thm]{Remark}
\newtheorem*{example*}{Example}

\numberwithin{equation}{section}

\newcommand{\ep}{\varepsilon}
\newcommand{\al}{\alpha}

\newcommand{\Om}{\Omega}
\newcommand{\Sig}{\Sigma}

\newcommand{\abs}[1]{\lvert#1\rvert}

\newcommand{\bigabs}[1]{\bigl\lvert#1\bigr\rvert}

\newcommand{\Bigabs}[1]{\Bigl\lvert#1\Bigr\rvert}

\newcommand{\bs}{\backslash}

\newcommand{\ol}{\overline}

\newcommand{\spn}{\operatorname{span}}

\newcommand{\co}{\operatorname{co}}
\newcommand{\so}{\operatorname{so}}

\newcommand{\sgn}{{\operatorname{Sign}}\,}
\newcommand{\var}{{\operatorname{var\,}}}

\DeclareMathOperator{\Span}{Span}

\newcommand{\N}{{\mathbb N}}
\newcommand{\R}{{\mathbb R}}

\newcommand{\Q}{{\mathbb Q}}

\newcommand{\E}{{\mathbb E}}
\newcommand{\m}{{\mathrm m}}
\newcommand{\ZL}{{\mathbb{L}^0}}
\newcommand{\OL}{{\mathbb{L}^1}}
\newcommand{\bP}{{\mathbb P}}

\newcommand{\cB}{{\mathcal B}}

\newcommand{\cL}{{\mathcal L}}
\newcommand{\cS}{{\mathcal S}}
\newcommand{\cA}{{\mathcal A}}
\newcommand{\cX}{{\mathcal X}}

\newcommand{\cU}{{\mathcal U}}
\newcommand{\cV}{{\mathcal V}}
\newcommand{\cW}{{\mathcal W}}

\newcommand{\cK}{{\mathcal K}}
\newcommand{\cC}{{\mathcal C}}

\newcommand{\one}{\mathbbm{1}}

\title{On local convexity in $\mathbb{L}^0$ and switching probability measures}

\author[N.~Gao]{Niushan Gao}
\address{Department of Mathematics, Ryerson University, 350 Victoria Street, Toronto, Canada M5B2K3}
\email{niushan@ryerson.ca}

\author[D.~Leung]{Denny H.~Leung}
\address{Department of Mathematics, National University of Singapore, Singapore 117543}
\email{matlhh@nus.edu.sg}

\author[F.~Xanthos]{Foivos Xanthos}
\address{Department of Mathematics, Ryerson University, 350 Victoria Street, Toronto, Canada M5B2K3}
\email{foivos@ryerson.ca}

\thanks{The first and third authors acknowledge support of NSERC Discovery Grants. The second author is partially supported by AcRF grant R-146-000-242-114.}

\keywords{Equivalent probability measures, uniformly integrable, locally convex, uniformly locally convex-solid, convergence in probability, measure-free}

\subjclass[2010]{46A55, 46E30, 46A16, 	 60A10,  46N30}

\date{\today}

\begin{document}

\maketitle
\begin{abstract}In the paper, we investigate the following fundamental question. For a set $\cK$ in $\ZL(\bP)$, when does there exist an equivalent probability measure $\Q$ such that $\cK$ is uniformly integrable in $\OL(\Q)$. Specifically, let $\cK$ be a convex bounded positive set in $\OL(\bP)$. Kardaras \cite{K} asked the following two questions: (1) If the relative $\ZL(\bP)$-topology is locally convex on $\cK$, does there exist $\Q\sim \bP$ such that the $\ZL(\Q)$- and $\OL(\Q)$-topologies agree on ${\cK}$? (2) If $\cK$ is closed in the $\ZL(\bP)$-topology and there exists $\Q\sim \bP$ such that the $\ZL(\Q)$- and $\OL(\Q)$-topologies agree on $\cK$, does there exist $\Q'\sim \bP$ such that $\cK$ is $\Q'$-uniformly integrable? In the paper, we show that, no matter $\cK$ is positive or not, the first question has a negative answer in general and  the second one has a positive answer. In addition to answering these questions, we establish probabilistic and
topological characterizations of existence of $\Q\sim\bP$ satisfying these desired properties. We also investigate the peculiar effects of $\cK$ being positive.
\end{abstract}

\section{Introduction}\label{intro}
The Fundamental Theorem of Asset Pricing establishes the prominent importance of working under an equivalent probability measure $\Q$ relative to the original physical probability measure $\bP$.
It is henceforth of great interest to study  how certain analytical and probabilistic properties of a set behave  when the underlying probability measure is switched from one to another.
This line of research can be traced back to the remarkable work Brannath and Schachermayer \cite{BS} and is significantly expanded in  two recent papers  Kardaras and \v{Z}itkovi\'{c} \cite{KZ} and   Kardaras \cite{K}. It turns that local convexity of the topology of convergence in probability plays an important role.

Throughout the paper, let $(\Om,\Sig,\bP)$ stand for a nonatomic probability space. Let $\ZL(\bP):=\ZL(\Om,\Sig,\bP)$ be the space of all random variables modulo a.s.-equality. By the $\ZL(\bP)$-topology, we refer to the topology of convergence in the probability measure $\bP$.  A probability measure $\Q$ on $(\Om,\Sig)$ is {\em equivalent to} $\bP$ if $\Q$ and $\bP$ are mutually absolutely continuous with respect to one another. In this case, we write $\Q \sim \bP$. It is well-known that if $\Q\sim \bP$, then $\ZL(\Q)=\ZL(\bP)$ and the $\ZL(\Q)$- and $\ZL(\bP)$-topologies coincide.

Given a sequence $(X_n)$ in $\ZL(\bP)$, a {\em forward convex combination (FCC)} of $(X_n)$ is a sequence $(Y_k)$ such that $Y_k \in \co(X_n)^\infty_{n=k}$ for each $k\in \N$. Here $\co(\cA)$ is the convex hull of a set $\cA$.  For a convex set $\cK$ in $\ZL(\bP)$ and $X\in \cK$, we say that the (relative) $\ZL(\bP)$-topology on $\cK$ is \emph{locally convex at $X$} if for any $\ZL(\bP)$-neighborhood $\mathcal{U}$ of $0$, there exists a convex neighborhood $\cW$ of $X$ in the relative  $\ZL(\bP)$-topology on $\cK$ such that $\cW\subset (X+\cU)\cap \cK$,
or equivalently, there exists a convex subset $\cW' $ of $\cU$ containing $0$ such that $(X+\cW')\cap \cK $ is a neighborhood of $X$ in the relative  $\ZL(\bP)$-topology on $\cK$ (e.g., taking $\cW'=\cW-X$, and $\cW=(X+\cW')\cap \cK$, conversely). It is  easily seen to be equivalent to that if $(X_n)$ is a sequence in $\cK$ that converges to $X$ in probability, then every FCC of $(X_n)$ also converges to $X$ in probability.
We say that the $\ZL(\bP)$-topology is \emph{locally convex on $\cK$} if it is locally convex at every point of $\cK$.

The following theorem is part of the main result  in \cite{KZ}.

\begin{thm}[\cite{KZ}]\label{t1}
Let $(X_n)$ be a sequence in $\mathbb{L}^0_+(\bP)$ that converges in probability to a random variable $X\in \mathbb{L}^0_+(\bP)$.
The following are equivalent.
\begin{enumerate}
\item\label{t11} Every FCC of $(X_n)$ converges to $X$ in probability.
\item\label{t12} The $\ZL(\bP)$-topology is locally convex on the set $\cK = \co\big((X_n)_{n=1}^\infty\cup \{X\}\big)$.
\item\label{t13} The $\ZL(\bP)$-topology is locally convex on the set $\ol{\cK}$, where the closure is  taken in $\ZL(\bP)$ with respect to the $\ZL(\bP)$-topology.
\item\label{t14} There exists $\Q\sim \bP$ such that the $\ZL(\Q)$- and $\OL(\Q)$-topologies agree on $\ol{\cK}$.
\end{enumerate}
\end{thm}

Theorem~\ref{t1} is extended in Kardaras \cite{K}. We say that a set $\cA$ in $\mathbb{L}^0_+(\bP)$ is {\em positive solid} if $Y\in \cA$ whenever there exists $X\in \cA$ such that $0\leq Y\leq X$.
A subset $\cA $ in $\mathbb{L}^0_+(\bP)$ is {\em bounded in probability} if $\sup_{X\in \cA}\bP(\abs{X} >n)\longrightarrow 0$ as $n\longrightarrow \infty$.

\begin{thm}[\cite{K}]\label{t2}
Let $\cK$ be a convex, positive solid set in $\mathbb{L}^0_+(\bP)$ that is bounded in probability.  The following are equivalent.
\begin{enumerate}
\item\label{t22} The $\ZL(\bP)$-topology on $\cK$ is locally convex at $0$.
\item\label{t23} The $\ZL(\bP)$-topology is locally convex on $\cK$.
\item\label{t24} There exists $\Q\sim \bP$ such that the $\ZL(\Q)$- and $\OL(\Q)$-topologies agree on $\cK$.
\item\label{t25} There exists $\Q\sim \bP$ such that $\cK$ is $\Q$-uniformly integrable.
\end{enumerate}
\end{thm}

Connections of these results to Mathematical Finance and Economics are also made in \cite{KZ,K}.
We refer to the references therein for further connections.

Clearly, \eqref{t25}$\implies$\eqref{t24}$\implies$\eqref{t23} in Theorem~\ref{t2} hold for an arbitrary set $\cK$ in $\ZL(\bP)$. The following example, however, shows that Conditions \eqref{t24} and \eqref{t25} do not necessarily agree for any convex sets in $\mathbb{L}^0_+(\bP)$ that are bounded in probability.

\begin{example*}[\cite{K}] Let $\cK= \big\{X\in \mathbb{L}^0_+(\bP): \E[X] =1\big\}$.  Then $\cK$ is a convex set in $\mathbb{L}^0_+(\bP)$ that is bounded in probability.  It is well-known that the $\ZL(\bP)$- and $\OL(\bP)$-topologies agree on $\cK$.  However, there is no $\Q\sim \bP$ such that $\cK$ is $\Q$-uniformly integrable.
\end{example*}

In view of these results, the following questions were raised in \cite{K}.
Let $\cK$ be a convex  set in $\mathbb{L}^0_+(\bP)$ that is bounded in probability.
\begin{enumerate}
\item[(Q1+)] Is it true that if the $\ZL(\bP)$-topology is locally convex on $\cK$, then there exists $\Q\sim \bP$ such that the $\ZL(\Q)$- and $\OL(\Q)$-topologies agree on $\cK$?
\item[(Q2+)] Assume that $\cK$ is also closed in $\ZL(\bP)$ with respect to the $\ZL(\bP)$-topology.  If there exists $\Q\sim \bP$ such that the $\ZL(\Q)$- and $\OL(\Q)$-topologies agree on $\cK$, does there exist $\Q'\sim \bP$ such that $\cK$ is $\Q'$-uniformly integrable?
\end{enumerate}
The ``+'' signs in the labels above remind us that these questions concern {\em positive} sets.

Brannath and Schachermayer \cite{BS} showed that if $\cK$ is a convex positive set in $\ZL(\bP)$ that is bounded in probability, then there exists $\bP'\sim \bP$ such that $\cK$ is bounded in $\OL(\bP')$.  Thus we may  assume that $\cK$ is bounded in $\OL(\bP)$ in the first place.
Hence we may ask the preceding questions for arbitrary convex bounded sets in $\OL(\bP)$.  We will refer to these questions as (Q1) and (Q2), respectively. The validity of Theorem \ref{t1} for suitable nonpositive sequences was alluded to in \cite[Remark 1.6]{KZ}.

\bigskip

We now describe the contributions of this paper with regard to the questions raised above.
First, it is shown that (Q2) and hence (Q2+) have positive solutions.
Precisely,

\begin{thm}\label{t3}
Let $\cK$ be a convex bounded subset of $\OL(\bP)$ that is  closed in the $\ZL(\bP)$-topology.  The following are equivalent.
\begin{enumerate}
\item There exists $\Q\sim \bP$ such that the $\ZL(\Q)$- and $\OL(\Q)$-topologies agree on $\cK$.
\item There exists $\Q\sim \bP$ such that $\cK$ is $\Q$-uniformly integrable.
\end{enumerate}
\end{thm}

Recall that a set $\cK$ in a vector lattice is {\em solid} if $Y\in \cK$ whenever there exists $X\in \cK$ such that $\abs{Y}\leq \abs{X}$.  Let $\cK$ be a convex bounded set in $\OL(\bP)$ and let $\mathcal{S}$ be a nonempty subset of $\cK$.
We say that the $\ZL(\bP)$-topology on $\cK$ is {\em uniformly locally convex-solid on $\mathcal{S}$} if for any $\ZL(\bP)$-neighborhood $\mathcal{U}$ of $0$, there exists a convex-solid set $\cW\subseteq \cU$ such that  $(X+\cW)\cap \cK$ is a neighborhood of $X$ in the relative $\ZL(\bP)$-topology on $\cK$, for every $X\in \mathcal{S}$.
If $\mathcal{S} = \{X\}$ is a singleton set, then we simply say that the $\ZL(\bP)$-topology on $\cK$ is locally convex-solid at $X$.  With this terminology, we obtain an intrinsic topological characterization of Condition \eqref{t23} of Theorem \ref{t2}.

\begin{thm}\label{t4}
Let $\cK$ be a convex bounded set in $\OL(\bP)$.  The following are equivalent.
\begin{enumerate}
\item The $\ZL(\bP)$-topology on $\cK$ is {uniformly locally convex-solid on $\cK$}.
\item There exists $\Q\sim \bP$ such that the $\ZL(\Q)$- and $\OL(\Q)$-topologies agree on $\cK$.
\end{enumerate}
\end{thm}

For a general convex bounded set  $\cK$  in $\OL(\bP)$, the condition that the $\ZL(\bP)$-topology on $\cK$ is  uniformly locally convex-solid on $\cK$ is genuinely stronger than the plain local convexity.  That is, (Q1) has a negative solution in general.

\begin{thm}[Example A]\label{cea}
There exists a convex bounded circled set $\cK$ in $\OL[0,1]$ that is $\ZL[0,1]$-compact, such that the $\ZL[0,1]$-topology on $\cK$ is locally convex but there does not exist a probability measure $\Q$ on $[0,1]$, equivalent to the Lebesgue measure, such that the $\ZL(\Q)$- and $\OL(\Q)$-topologies agree on $\cK$.
\end{thm}

The construction of the example is based on an example of Pryce \cite{P}.
However, the set $\cK$ in Example A is not contained in $\mathbb{L}_+^1[0,1]$.
Nevertheless, it turns out that  (Q1+) has a negative answer in general as well.

\begin{thm}[Example B]\label{ceb}
There exist a nonatomic probability space $(\Omega,\Sigma,\bP)$ and  a convex bounded  set $\cK$ in $\mathbb{L}_+^1(\bP)$  such that the $\ZL(\bP)$-topology on $\cK$ is locally convex but there does not exist $\Q\sim \bP$ such that the $\ZL(\Q)$- and $\OL(\Q)$-topologies agree on $\cK$.
\end{thm}

\smallskip

Unlike in Example A, the set $\cK$ in Example B, as well as the underlying measure space,  is nonseparable, neither is it closed  in the $\ZL(\bP)$-topology.
Hence, the following modifications of (Q1+) are still open.

\begin{enumerate}
\item[(Q1')]  Let $\cK$ be a convex bounded set in $\mathbb{L}^1_+(\bP)$.
Assume that the $\ZL(\bP)$-topology is locally convex on $\cK$.
Is it true that if  $\cK$ is closed, or separable, in $\ZL(\bP)$, or if both conditions hold (in particular, if $\cK$ is compact in $\ZL(\bP)$), then there exists $\Q\sim \bP$ such that the $\ZL(\Q)$- and $\OL(\Q)$-topologies agree on $\cK$?
\end{enumerate}
Note that the $\ZL(\bP)$-topology is metrizable and thus compact sets in $\ZL(\bP)$ are both closed and separable.
Note also that if $\cK$ is a separable set in $\mathbb{L}^0_+(\bP)$, then there is a non-atomic separable sub-$\sigma$-algebra $\Sigma'$ of $\Sigma$ such that $\cK$ is $\Sigma'$-measurable.

\smallskip

Finally, concerning the problems (Q1+) and (Q1'),
we have the  following result in the positive direction that is somewhat surprising and complements Theorem \ref{t1}.

\begin{thm}\label{t5}
Let $(X_n)$ be a bounded sequence in $\mathbb{L}^1_+(\bP)$ and let $\cK = \co(X_n)$.  The following are equivalent.
\begin{enumerate}
\item The $\ZL(\bP)$-topology is locally convex on $\cK$.
\item There exists $\Q\sim \bP$ such that the $\ZL(\Q)$- and $\OL(\Q)$-topologies agree on $\cK$.
\end{enumerate}
\end{thm}

We also include alternative proofs of Theorems \ref{t1} and \ref{t2} in the spirit of the present paper as an appendix at the end.

\section{``De-switching" probability measures}

The main conditions of interest in Theorems \ref{t1} and \ref{t2} and in the questions (Q1) and (Q2) involve switching from a probability measure $\bP$ to an equivalent one.  It would be convenient to reformulate these conditions to remove the switching of probability measures.  We begin with a simple lemma that is essentially an exhaustion technique.

\begin{lem}\label{l2.0}
Let $\xi: \Sig \to \{0,1\}$ be a function such that $\xi(A) \geq \xi(B)$ if $A\subseteq B$ and that $\xi(A\cup B) =1$ if $\xi(A) = \xi(B) =1$.
Then there exists $C\in \Sig$ such that
\begin{equation}\label{e2.1}\bP(C) = \sup\big\{\bP(A): A \in \Sigma,\ \xi(A) = 1\big\}  \;\text{ and }\; \bP(A\bs C) = 0 \text{ if } \xi(A) =1.\end{equation}
\end{lem}

\begin{proof}
Define
\[ a = \sup\big\{\bP(A): A \in \Sigma,\ \xi(A) = 1\big\}\]
Choose a sequence $(A_n)$ in $\Sig$ such that $\xi(A_n) =1$ for all $n\in\N$ and $\bP(A_n)\longrightarrow a$.
Let $C = \cup_{n=1}^\infty A_n$.
Note that $\xi(\cup^n_{m=1}A_m) =1$ for all $n\in\N$.  Hence, $\bP(A_n) \leq \bP(\cup^n_{m=1}A_m) \leq a$ for all $n$.
It follows that $\bP(C) =\lim_n\bP(\cup^n_{m=1}A_m)=a$.
Suppose that $A\in \Sig$ and $\xi(A) =1$.
Since $A\bs C\subseteq A$, $\xi(A\bs C) \geq \xi(A) =1$, implying that $\xi(A\bs C)=1$.
If $\bP(A\bs C) >0$, we can choose $n\in\N$ such that $\bP(A_n) > a - \bP(A\bs C)$.
Since $A_n\subseteq C$, $A_n$ and $A\bs C$ are disjoint sets.
Thus,
\[ \bP\big(A_n \cup (A\bs C)\big) = \bP(A_n) + \bP(A\bs C) > a.\]
But we also have $\xi\big(A_n \cup (A\bs C)\big) = 1$ since $\xi(A_n) = \xi(A\bs C) = 1$.
This contradicts the choice of $a$.  Thus $\bP(A\bs C) = 0$, as desired.
\end{proof}

\begin{prop}\label{p2.1}
Let $\cK$ be a convex bounded subset of $\OL(\bP)$ and let $\mathcal{S}$ be a nonempty subset of $\OL(\bP)$.
The following are equivalent.
\begin{enumerate}
\item\label{p2.11} There exists $\Q \sim \bP$ such that if $(X_n)$ is a sequence in $\cK$ that converges in probability to some $X\in \cS$, then $(X_n)$ converges to $X$ in $\OL(\Q)$.
\item\label{p2.12} For any $\ep >0$, there exists a measurable set $A$ with $\bP(A) > 1-\ep$
such that if $(X_n)$ is a sequence in $\cK$ that converges in probability to some $X\in \cS$, then $\E_\bP\big[\abs{X_n-X}\one_A\big] \longrightarrow 0$.
\item\label{p2.13} For any measurable set $A$ with $\bP(A) >0$, there exists a measurable subset $B$ of $A$ with $\bP(B) >0$ such that  if $(X_n)$ is a sequence in $\cK$ that converges in probability to some $X\in \cS$, then $\E_\bP\big[\abs{X_n-X}\one_B\big] \longrightarrow 0$.
\end{enumerate}
\end{prop}

\begin{proof}
\eqref{p2.11}$\implies$\eqref{p2.12}.  Assume that \eqref{p2.11} holds. Note that $Y:=\frac{\mathrm{d}\Q}{\mathrm{d}\bP}>0$ a.s.  Given $\ep >0$, choose $r >0$ such that $A = \{Y \geq  r\}$ satisfies $\bP(A) > 1-\ep$.
Suppose that $(X_n)$ is a sequence in $\cK$ that converges in probability to some $X\in \cS$.
Then
\[ \E_\bP\big[\abs{X_n-X}\one_A\big] \leq \frac{1}{r} \E_\bP\big[\one_A\abs{X_n-X}Y\big] \leq \frac{1}{r}\E_\Q\big[\abs{X_n-X}\big] \longrightarrow 0.\]

\eqref{p2.12}$\implies$\eqref{p2.13}. Assume that \eqref{p2.12} holds and let  $A\in\Sigma$ be such that $\bP(A) >0$.
By (2), choose a measurable set $C$ with $\bP(C) > 1- \bP(A)$ such that if  $(X_n)$ is a sequence in $\cK$ that converges in probability to some $X\in \cS$, then $\E_\bP\big[\abs{X_n-X}\one_C\big]  \longrightarrow 0$.
Since $\bP(A) + \bP(C) > 1$, $\bP(A\cap C) > 0$.  Let $B = A \cap C$.  Then $B$ satisfies Condition \eqref{p2.13}.

\eqref{p2.13}$\implies$\eqref{p2.11}. Assume  that \eqref{p2.13} holds. Define a function $\xi:\Sig\to\{0,1\}$ as follows. Set  $\xi(A) =1$ if  for any sequence $(X_n)$ in $\cK$ that converges in probability to some $X\in \cS$, $\E_\bP\big[\abs{X_n-X}\one_A\big]\longrightarrow 0$, and $0$ otherwise.  It is clear that $\xi$ satisfies the hypotheses of Lemma \ref{l2.0}. By the lemma, there exists $C\in \Sig$ satisfying \eqref{e2.1}.
If $\bP(C^c) > 0$, then by assumption, there exists a measurable set $B\subseteq C^c$ such that $\bP(B) > 0$ and $\xi(B) =1$.  By \eqref{e2.1}, $0=\bP(B\backslash C)=\bP(B)$, where the second equality holds because $B\bs C=B$. This contradicts the choice of $B$.
Hence, $\bP(C) =1$.

Let $c = \sup_{X\in \cK}\E_\bP[\abs{X}]$ and let $\ep > 0$ be given.
Since $\bP(C) =1$, there is a sequence $(A_k)$ in $\Sig$ such that $\bP(A_k) \uparrow 1$ and that $\xi(A_k) =1$ for all $n\in\N$.  We may replace $A_k$ with $\cup^k_{j=1}A_j$, if necessary, to assume that $A_k \subseteq A_{k+1}$ for all $k\in\N$. We may also assume that $\Omega=\cup_{k=1}^\infty A_k$ since $\bP\big(\cup_{k=1}^\infty A_k\big)=1$.
Set $A_0 = \emptyset$ and define $Y$ to be $\frac{1}{2^k}$ on the set $A_{k} \bs A_{k-1}$ for any $k\in\N$.
Then $Y$ is strictly positive and $\Q \sim \bP$, where $\mathrm{d}\Q = \frac{Y}{\E_\bP[Y]}\mathrm{d}\bP$.
Suppose that $(X_n)$ is a sequence in $\cK$ that converges   in probability to some $X\in \cS$. By Fatou's Lemma, $\E_\bP[\abs{X}]\leq \liminf_n\E_\bP[\abs{X_n}]\leq c$.
For any $k$, $0 \leq Y \leq \frac{1}{2^{k+1}}$ on $A^c_k$.  Hence, for any $n,k\in\N$,
\[ \E_\bP\big[\abs{X_n-X}Y\one_{A_k^c}\big] \leq \frac{1}{2^{k+1}}\big(\E_\bP[\abs{X_n}]+\E_\bP[\abs{X}]\big)\leq \frac{c}{2^{k}}.
\]
Note that $Y\leq 1$ pointwise.
Thus, for all $n$ and $k$,
\begin{align*}
 \E_\Q\big[\abs{X_n-X}\big] = & \frac{1}{\E_\bP[Y]}\E_\bP\big[\abs{X_n-X}Y\one_{A_k}\big] +\frac{1}{\E_\bP[Y]}\E_\bP\big[\abs{X_n-X}Y\one_{A_k^c}\big]\\ \leq & \frac{1}{\E_\bP[Y]}\E_\bP\big[\abs{X_n-X}\one_{A_k}\big] +\frac{c}{2^{k}\E_\bP[Y]}.
\end{align*}
Since $\xi(A_k) =1$, $\E_\bP\big[\abs{X_n-X}Y\one_{A_k}\big]  {\longrightarrow} 0$ as $n\longrightarrow\infty$.
Therefore, $$\limsup_n \E_\Q\big[\abs{X_n-X}\big] \leq \frac{c}{2^{k}\E_\bP[Y]}$$ for any $k$, so that $\E_\Q\big[\abs{X_n-X}\big] \longrightarrow0$.
Condition \eqref{p2.11} thus holds for $\Q$ as chosen.
\end{proof}

Although not needed, we remark that $\frac{\mathrm{d}\Q}{\mathrm{d}\bP}$ is bounded for $\Q$ constructed above.

Before proceeding further, let us recall the well-known theorem of Koml\'{o}s \cite{Kom}.
The result is applied to prove the  crucial step
\eqref{p2.23}$\implies$\eqref{p2.24} in Proposition \ref{p2.2} below.

\begin{lem}[\cite{Kom}]\label{Kom}
Let $(X_n)$ be a bounded sequence in $\OL(\bP)$.  Then there exist a subsequence $(X_{n_k})$ of $(X_n)$ and a random variable $X\in \OL(\bP)$ such that for any further subsequence $(X_{n_{k_j}})$ of $(X_{n_k})$,
\[ \lim_m\frac{1}{m}\sum^m_{j=1}X_{n_{k_j}} = X\ a.s.\]
\end{lem}

\begin{prop}\label{p2.2}
Let $\cK$ be a convex bounded subset of $\OL(\bP)$.
The following are equivalent.
\begin{enumerate}
\item\label{p2.21} There exists $\Q \sim \bP$ such that $\cK$ is $\Q$-uniformly integrable.
\item\label{p2.22}  For any $\ep >0$, there exists a measurable set $A$ with $\bP(A) > 1-\ep$ such that if $(X_n)$ is a sequence in $\cK$ that is Cauchy in probability, then $\E_\bP\big[\abs{X_n-X_m}\one_A\big]\longrightarrow 0$ as $n,m\longrightarrow \infty$.
\item\label{p2.23} For any measurable set $A$ with $\bP(A) >0$, there exists a measurable subset $B$ of $A$ with $\bP(B) >0$ such that  if $(X_n)$ is a sequence in $\cK$ that is Cauchy in probability, then $\E_\bP\big[\abs{X_n-X_m}\one_B\big]\longrightarrow 0$ as $n,m\longrightarrow \infty$.
\item\label{p2.24} For any measurable set $A$ with $\bP(A) >0$, there exists a measurable subset $B$ of $A$ with $\bP(B) >0$ such that $\cK_B: = \{X\one_B: X\in \cK\}$ is $\bP$-uniformly integrable.
\end{enumerate}
\end{prop}

\begin{proof}
Let $\Q$ be a probability measure and suppose that $(X_n)$ is a sequence of random variables that is  Cauchy in probability and is $\Q$-uniformly integrable.  Then $(X_n)$ converges in probability to some $X\in \ZL(\Q)$.
Since $(X_n)$ is $\Q$-uniformly integrable, $X$ is $\Q$-integrable and
$(X_n)$ converges  to $X$ in $\OL(\Q)$.
Therefore, $(X_n)$ is $\OL(\Q)$-Cauchy.
Using this observation, the implications \eqref{p2.21}$\implies$\eqref{p2.22}$\implies$\eqref{p2.23} can be shown exactly as in the corresponding steps in Proposition \ref{p2.1}.

The proof of \eqref{p2.24}$\implies$\eqref{p2.21} is also similar to the proof of \eqref{p2.13}$\implies$\eqref{p2.11} in Proposition \ref{p2.1}.
Define $\xi: \Sig\to \{0,1\}$ by $\xi(A) =1$ if  $\cK_A$ is $\bP$-uniformly integrable.
Let $C$ be obtained by applying Lemma \ref{l2.0} to $\xi$.
It follows from the assumption \eqref{p2.24} that $\bP(C) =1$.
Take an increasing sequence of measurable sets $(A_k)$ such that $\xi(A_k) =1$ for all $k\in\N$, $\bP(A_k) \longrightarrow 1$, and $\Omega=\cup_{k=1}^\infty A_k$.
Set $A_0 = \emptyset$ and define $Y$ to be $\frac{1}{2^k}$ on the set $A_{k} \bs A_{k-1}$ for any $k\in\N$.
Let $\mathrm{d}\Q = \frac{Y}{\E_\bP[Y]}\mathrm{d}\bP$. Then  $\Q \sim \bP$.
We claim that $\cK$ is $\Q$-uniformly integrable.  Clearly, $\cK$ is bounded in $\OL(\Q)$. Set $c = \sup_{X\in \cK}\E_\bP[\abs{X}]$.
Let $\ep >0$ be given.
Choose $k$ large enough so that $\frac{c}{2^{k}\E_\bP[Y]}\leq \ep$.
Since $\xi(A_k) =1$, $\cK_{A_k}$ is $\bP$-uniformly integrable.
Therefore, there exists $\delta >0$ such that $$\sup_{X\in \cK}\E_\bP\big[\abs{X}\one_B\big] < \frac{\ep\E_\bP[Y]}{2}\;\text{ if }B\subseteq A_k\text{ and }\bP(B) < \delta.$$
Now, take any $A\in \Sig$ such that $\Q(A) < \frac{\delta}{2^k\E_\bP[Y]}$.
Let $B_1 = A\cap A_k$  and $B_2 = A\bs A_k$.
Since $Y \geq \frac{1}{2^k}$ on $A_k$,
$\Q(B_1) = \frac{1}{\E_\bP[Y]}\E_\bP[Y\one_{B_1}]\geq \frac{1}{2^k\E_\bP[Y]}\bP(B_1)$, so that
\[ \bP(B_1) \leq 2^k\E_\bP[Y]\Q(B_1) \leq  2^k\E_\bP[Y]\Q(A) <\delta.\]
Thus, for any $X\in \cK$, $\E_\bP\big[\abs{X}\one_{B_1}\big]  < \frac{\ep\E_\bP[Y]}{2}$.  Moreover, note that $0 \leq Y\leq \frac{1}{2^{k+1}}$ on $A_k^c\supset B_2$.  Thus, if $X\in \cK$, then
\begin{align*}
\E_\Q\big[\abs{X}\one_A\big]  = & \frac{1}{\E_\bP[Y]}\E_\bP\big[\abs{X}Y\one_{B_1}\big]+\frac{1}{\E_\bP[Y]}
\E_\bP\big[\abs{X}Y\one_{B_2}\big] \\
\leq & \frac{1}{\E_\bP[Y]}\E_\bP\big[\abs{X}\one_{B_1}\big]+\frac{c}{\E_\bP[Y]2^{k+1}}\\
\leq &\frac{\ep}{2}+\frac{\ep}{2}=\ep.
\end{align*}
This proves that $\cK$ is $\Q$-uniformly integrable, and thus \eqref{p2.24}$\implies$\eqref{p2.21}.

Assume that \eqref{p2.23} holds.  Let $A$ be a measurable set with $\bP(A) >0$.  Choose a measurable subset $B$ of $A$ with $\bP(B) >0$ as in Condition \eqref{p2.23}.  We aim to show that $\cK_B$ is $\bP$-uniformly integrable.
Suppose the contrary.
By \cite[Theorem 5.2.9]{AK:06}, there exist a real number $c' > 0$ and  a sequence $(X_n)$ in $\cK$ such that for any $n\in\N$ and any real numbers $a_1,\dots,a_n$,
$$\E_\bP\Big[\Bigabs{\sum_{k=1}^na_kX_k\one_B}\Big]\geq c'\sum_{k=1}^n\abs{a_k}.$$
Applying Koml\'{o}s' Theorem and relabeling, we may assume that the arithmetic means of $(X_n)$ converge to some $X\in \ZL(\mathbb{P})$ a.s. Put
$$Y_n=\frac{1}{2^n}\sum_{k=1}^{2^n}X_k.$$
Clearly, $(Y_n)\subset \cK$ is Cauchy in probability, and thus by choice of $B$, $\big(Y_n\one_B\big)$ is Cauchy in $\OL(\bP)$. On the other hand, whenever $n>m$,
\begin{align*}
\E_\bP\big[\bigabs{Y_n\one_B-Y_m\one_B}\big]=&  \E_\bP\Big[\Bigabs{\sum_{k=1}^{2^m}\big(\frac{1}{2^n}-\frac{1}{2^m}\big)X_k
\one_B
+\sum_{k=2^m+1}^{2^n}\frac{1}{2^n}X_k\one_B}\Big]\\
&\geq c'\Big(\sum_{k=1}^{2^m}\big(\frac{1}{2^m}-\frac{1}{2^n}\big)+
\sum_{k=2^m+1}^{2^n}\frac{1}{2^n}\Big)\\
&= c'\Big(1-\frac{2^m}{2^n}+\frac{2^n-2^m}{2^n}\Big)
\geq c'.
\end{align*}
This contradiction completes the proof.
\end{proof}

The next corollary clarifies the relationship between Conditions \eqref{t24} and \eqref{t25} of Theorem \ref{t2} and answers the questions (Q2) and (Q2+) in the positive.

\begin{cor}\label{c2.3}
Let $\cK$ be a convex bounded subset of $\OL(\bP)$.
The following are equivalent.
\begin{enumerate}
\item\label{c2.31} There exists $\Q \sim \bP$ such that $\cK$ is $\Q$-uniformly integrable.
    \item\label{c2.32} There exists $\Q \sim \bP$ such that $\ol{\cK}$ is $\Q$-uniformly integrable, where the closure is taken in the $\ZL(\bP)$-topology.
\item\label{c2.33} There exists $\Q\sim \bP$ such that the $\ZL(\Q)$- and $\OL(\Q)$-topologies agree on $\ol{\cK}$.
\end{enumerate}
\end{cor}

\begin{proof}
Let $\Q$ be as given in Condition \eqref{c2.31}.
Let $\ep >0$ be given.  Then there exists $\delta > 0$ such that $\E_\Q\big[\abs{X}\one_A\big] <\ep$ if $X\in \cK$ and $\Q(A) <\delta$.
Suppose that $X\in \ol{\cK}$ and $\Q(A) < \delta$.  Choose a sequence $(X_n)$ in $\cK$ that converges to $X$ in $\ZL(\bP)$.  Then it also converges to $X$ in $\ZL(\Q)$.
Thus by Fatou's Lemma, $$\E_\Q\big[\abs{X}\one_A\big]\leq \liminf_n\E_\Q\big[\abs{X_n}\one_A\big]\leq \ep.$$
Moreover, since $\cK$ is bounded in $\OL(\Q)$, a similar argument shows that $\ol{\cK}$ is also bounded in $\OL(\Q)$. Thus $\ol{\cK}$ is $\Q$-uniformly integrable.
This proves \eqref{c2.31}$\implies$\eqref{c2.32}.

The implication \eqref{c2.32}$\implies$\eqref{c2.33} is clear.

Assume that \eqref{c2.33} holds.  We apply Proposition \ref{p2.1} to $\ol{\cK}$ with $\cS = \ol{\cK}$.
For any $\ep >0$, we obtain a measurable set $A$ with $\bP(A) > 1-\ep$ such that if $(X_n)$ is a sequence in $\ol{\cK}$ that converges to  $X\in \ol{\cK}$ in probability, then $\E_\bP\big[\abs{X_n-X}\one_A\big]\longrightarrow0$.
Now let  $(X_n)$ be any sequence  in $\ol{\cK}$ that is  Cauchy in probability. Since $\ol{\cK}$ is closed in $\ZL(\bP)$, $(X_n)$ converges in probability to some $X\in \ol{\cK}$.  Therefore, $\E_\bP\big[\abs{X_n-X}\one_A\big]\longrightarrow0$, and thus  $\E_\bP\big[\abs{X_n-X_m}\one_A\big]\longrightarrow0$ as $n,m\to\infty$.  We have thus verified Condition \eqref{p2.22} of Proposition \ref{p2.2} for the set $\ol{\cK}$ and therefore for the set $\cK$.  By the same result, Condition \eqref{c2.31} holds.
This proves \eqref{c2.33}$\implies$\eqref{c2.31}.
\end{proof}

Notice that Theorem \ref{t3} is an immediate consequence of Corollary \ref{c2.3}.

\section{Uniformly locally convex-solid topologies}

In this section, we first characterize topologically the $\OL(\bP)$-bounded convex sets $\cK$ on which  there exists $\Q\sim\bP$ such that the $\ZL(\Q)$- and $\OL(\Q)$-topologies agree.
The condition, as indicated in Theorem~\ref{t4}, is precisely that the relative $\ZL(\bP)$-topology is uniformly locally convex-solid on $\cK$, which is introduced in Section \ref{intro}. We begin our exploration with a result of the Hahn-Banach   theorem spirit.
Similar results of this type have been achieved in a recent paper \cite{GLX:18}, where the authors established a ``localized'' Hahn-Banach theorem on a vector space and applied it to study the uo-dual of a Banach lattice, resulting in a very transparent proof of Theorem~\ref{t2}. The following result is an extension of their approach, embracing solidity.

\begin{prop}\label{sep2}
Let $\cK$ be a convex set in $\OL(\bP)$ and let $\cS$ be a nonempty subset of $\cK$.
Suppose that the relative  $\ZL(\bP)$-topology on $\cK$ is uniformly locally convex-solid on $\cS$.
Then for any  measurable set $A$ with $\bP(A) >0$, there exists a nonzero random variable $Y\in \mathbb{L}^\infty_+(\bP)$, supported in $A$, such that
$\E_\bP\big[\abs{X_n-X}Y\big]\longrightarrow 0$ for any sequence $(X_n)$ in $\cK$ that converges to some $X\in \cS$ in probability.
\end{prop}

\begin{proof}Since $\one_A\neq0$, we can inductively choose $\ZL(\bP)$-neighborhoods $\cV$ and $\cU_k$ of $0$ such that
\[ (\cV+ \cU_1+\cU_1)\cap (\one_A+\cV+\cU_1+\cU_1) = \emptyset, \]
\[ \cU_k + k\cU_k \subseteq \cU_{k-1}, \quad \text{ if $k > 1$}.
\]
It is easily verified by induction that
\begin{equation}\label{disjoint}
(\cV+ \cU_1 + 2\cU_2+ \cdots +k\cU_k+\cU_k) \cap (\one_A+\cV + \cU_1+ 2\cU_2+ \cdots +k\cU_k+\cU_k) = \emptyset \end{equation}
for all $k \geq 1$.
For each $k\geq 1$, choose a convex solid set $\cW_k \subseteq \cU_k$  such that, for any $X\in \cS$, $(X+\cW_k)\cap \cK$ is  a neighborhood of $X$ in the relative $\ZL(\bP)$-topology on $\cK$.  Replace $\cW_k$ by $\cW_k\cap \OL(\bP)$, if necessary, to assume that $\cW_k \subseteq \OL(\bP)$.
Let $\mathcal{B}_{\OL(\bP)}$ be the closed unit ball of $\OL(\bP)$.
Since $\cV$ is an $\ZL(\bP)$-neighborhood of $0$, there exists $r>0$ such that $r\mathcal{B}_{\OL(\bP)} \subseteq \cV$.
Set
\[ \mathcal{C}_k = r\mathcal{B}_{\OL(\bP)} + \cW_1 + 2\cW_2 + \cdots + k\cW_k \; \text{ for each $k$.}\]
Since $\mathcal{B}_{\OL(\bP)}$ and each $\cW_k$ are convex and solid in $\OL(\bP)$, $\mathcal{C}_k$ is also convex and solid in $\OL(\bP)$ for all $k\in\N$ (solidity easily follows from the Riesz Decomposition Theorem \cite[Theorem 1.13]{AB:06}). Moreover, $k\cW_k \subseteq \mathcal{C}_k \subseteq \mathcal{C}_{k+1}$ for all $k$.
Let $$\mathcal{C} = \cup_{k=1}^\infty \mathcal{C}_k.$$  Then it is easily seen that $\mathcal{C}$ is  a convex  solid set in $\OL(\bP)$.  Since $r\mathcal{B}_{\OL(\bP)} \subseteq \mathcal{C}$, $\mathcal{C}$ absorbs $\OL(\bP)$, that is, for any $X\in \OL(\bP)$, $X\in t\mathcal{C}$ whenever $\abs{t}\geq t_0$ for some $t_0\in\R$.

Let $\rho:\OL(\bP) \to \R$ be the Minkowski functional for  $\mathcal{C}$ defined by $$\rho(X)=\inf\big\{\lambda>0:\frac{X}{\lambda}\in \mathcal{C}\big\}.$$
Then $\rho$ is a seminorm on $\OL(\bP)$ (see, e.g., \cite[Theorem~1.35]{R:91}). Note that since $\cC_k$ is solid, $\cC_k=-\cC_k$, and thus $\cC_k-\cC_k=2\cC_k$ by convexity of $\cC_k$. Since $ \mathcal{C}_k \subseteq \cV+ \cU_1 + 2\cU_2+ \cdots +k\cU_k+\cU_k$,  $\mathcal{C}_k \cap (\one_A+\mathcal{C}_k) = \emptyset$ by \eqref{disjoint}, so that $\one_A\notin \mathcal{C}_k-\mathcal{C}_k=2\mathcal{C}_k$ for any $k\in\N$.
Thus $\frac{\one_A}{2}\notin \mathcal{C}$, and therefore, $\rho(\one_A) \geq 2$. Define $\phi_0: \Span\{\one_A\} \to \R$ by $\phi_0(\al \one_A) = 2 \al$.
Then $\phi_0$ is a linear functional on $\Span\{\one_A\}$ and $\phi_0(\alpha \one_A)=2\alpha\leq 2\abs{\alpha}\leq \rho(\alpha \one_A)$ for any $\alpha\in\R$.
By the vector-space version of Hahn-Banach Theorem (see, e.g., \cite[Theorem~3.2]{R:91}), there is a linear functional $\phi:\OL(\bP) \to \R$ that extends $\phi_0$ and such that $\phi(X) \leq \rho(X)$ for all $X\in \OL(\bP)$. In particular, $\phi(\one_A) =\phi_0(\one_A)=2\neq 0$.
As $r\mathcal{B}_{\OL(\bP)}\subseteq \cC$, $\phi$ is bounded with respect to the $\OL(\bP)$-norm.  Hence there exists $Y_0\in \mathbb{L}^\infty(\bP)$ such that $$\phi(X)  =\E_\bP[XY_0]$$ for all $X\in \OL(\bP)$.
In particular, $\E[\one_AY_0] \neq 0$ and hence $Y_0\one_A \neq 0$.
Set $Y =|Y_0|\one_A$.  Then $Y$ is a nonzero random variable in ${\mathbb{L}}^\infty_+(\bP)$ that is supported in $A$.

Suppose that $(X_n)$ is a sequence in $\cK$ that converges to some $X\in \cS$ in the $\ZL(\bP)$-topology.
Pick any $k\in \N$.  Since $(X+\cW_k)\cap \cK$ is a neighborhood of $X$ with respect to the relative $\ZL(\bP)$-topology on $\cK$, there exists $N \in \N$ such that
$X_n \in X +\cW_k$ if $n\geq N$.  Consider any $n\geq N$. As $\cW_k$ is solid, $\abs{X_n-X}\one_A\sgn (Y_0) \in \cW_k$. Hence $k\abs{X_n-X}\one_A\sgn (Y_0)\in k\cW_k \subseteq \mathcal{C}_k\subseteq \mathcal{C}$.
Therefore, $\rho\big(k\abs{X_n-X}\one_A\sgn (Y_0)\big) \leq 1$.
It follows that $\E_\bP\big[ k\abs{X_n-X}\one_A\abs{Y_0}\big] = \phi\big(k\abs{X_n-X}\one_A\sgn (Y_0))\leq 1$, and thus
$$\E_\bP\big[ \abs{X_n-X}\one_A\abs{Y_0}\big]\leq \frac{1}{k}\quad\text{for any }n\geq N.$$
This proves that $\E_\bP\big[ \abs{X_n-X}\one_A\abs{Y_0}\big]\longrightarrow 0$.
\end{proof}

We now prove a slightly stronger version of Theorem \ref{t4}. In the proof below, we use the specific metric $d(X,Y) = \E_\bP[\abs{X-Y}\wedge \one]$ to generate the $\ZL(\bP)$-topology.  Balls $\cB(X,r)$ are taken with respect to this metric for any $X\in \ZL(\bP)$ and any $r>0$.

\begin{thm}\label{lcsolid}
Let $\cK$ be a convex bounded subset of $\OL(\bP)$ and let $\cS$ be a nonempty subset of $\cK$.  The following are equivalent.
\begin{enumerate}
\item\label{lcsolid1} The relative $\ZL(\bP)$-topology on $\cK$ is uniformly  locally convex-solid on $\cS$.
\item\label{lcsolid2} There exists $\Q \sim \bP$ such that if $(X_n)$ is a sequence in $\cK$ that converges in probability to some $X\in \cS$, then $(X_n)$ converges to $X$ in $\OL(\Q)$.
\end{enumerate}
\end{thm}

\begin{proof}
Assume that \eqref{lcsolid1} holds.
Let $A$ be a $\bP$-measurable set with $\bP(A) >0$.
By Proposition \ref{sep2}, there exists a nonzero random variable $Y\in \mathbb{L}^\infty_+(\bP)$, supported in $A$, such that
$\E_\bP[\abs{X_n-X}Y] \longrightarrow 0$ for any sequence $(X_n)$ in $\cK$ that converges to some $X\in \cS$ in probability.
There exists $r >0$ such that $B = \{Y \geq r\}$ has positive $\bP$-measure.
By choice, $B\subseteq A$.
Also,
\[ \E_\bP\big[\abs{X_n-X}\one_B\big]\leq \frac{1}{r}\E_\bP[\abs{X_n-X}Y]\longrightarrow 0
\]
if  $(X_n)$ is a sequence in $\cK$ that converges to some $X\in \cS$ in probability. Thus Condition \eqref{p2.13} of  Proposition \ref{p2.1} is satisfied, and hence by the same result, \eqref{lcsolid2} holds.

Assume that \eqref{lcsolid2} holds.
Let $\Q$ be given as in Condition \eqref{lcsolid2}, and write $Y=\frac{\mathrm{d}\Q}{\mathrm{d}\bP}$.
Let $\cU$ be an $\ZL(\bP)$-neighborhood of $0$, and let $r>0$ be such that $\cB(0,r) \subseteq \cU$.
Choose $\delta >0$ such that $\bP(Y< \delta) < \frac{r}{2}$ and let $s = \frac{r\delta}{2}$.
Let
\[ \cW = \big\{X \in \OL(\bP): \E_\Q[\abs{X}]=\E_\bP[\abs{X}Y] < s\big\}.\]
Obviously, $\cW$ is a convex solid set in $\OL(\bP)$.
If $X\in \cW$, then
\begin{align*}
d(X,0) =& \E_\bP[\abs{X}\wedge \one]\leq
\E_\bP\big[\abs{X}\one_{\{Y\geq \delta\}}\big]+\E_\bP\big[ \one_{\{Y<\delta\}}\big]\\
\leq &\frac{1}{\delta}\E_\bP\big[\abs{X}Y\one_{\{Y\geq \delta\}}\big] + \bP(Y < \delta)\\
<& \frac{s}{\delta} + \frac{r}{2}= r.
\end{align*}
This proves that $\cW \subseteq \cB(0,r) \subseteq \cU$.
Pick any $X\in\cS$.
Recall that  the $\ZL(\Q)$- and $\ZL(\bP)$-topologies agree on $\ZL(\Q) = \ZL(\bP)$.
It follows easily from the assumption \eqref{lcsolid2} that
there exists $t >0$ such that $\E_\Q[\abs{X'-X}]<s$ for all $X'\in \cB(X,t) \cap \cK$.
This means that
\[ \cB(X,t)\cap \cK \subseteq (X+\cW)\cap \cK,\]
and hence $(X+\cW)\cap \cK$ is a
neighborhood of $X$ in the relative $\ZL(\bP)$-topology on $\cK$. Thus  \eqref{lcsolid1} holds, and the proof is completed.
\end{proof}

Clearly, Theorem \ref{t4} is an immediate consequence of  Theorem \ref{lcsolid} by taking $\cS = \cK$.
Combining Theorem \ref{t4} and Corollary \ref{c2.3}, we also obtain a measure-free characterization of uniform integrability in the sense of Kardaras \cite{K}.

\begin{cor}\label{c3.4}
Let $\cK$ be a convex bounded subset of $\OL(\bP)$. The following are equivalent.
\begin{enumerate}
\item The relative $\ZL(\bP)$-topology on $\ol{\cK}$ is uniformly  locally convex-solid on $\ol{\cK}$, where the closure is taken in the $\ZL(\bP)$-topology.
\item There exists $\Q \sim \bP$ such that $\cK$ is $\Q$-uniformly integrable.
\end{enumerate}
\end{cor}

\smallskip

Example A, to be presented in the next section, shows that for the relative $\ZL(\bP)$-topology on $\cK$, being uniformly locally convex-solid is  strictly stronger than being only local convex, for a general convex bounded set $\cK$ in $\OL(\bP)$. However, we now show that in the presence of positivity, the equivalence of these two conditions may be established for some sets $\cK$.

The main additional feature that positivity brings in is the following.

\begin{lem}[{\cite[Lemma 2.4]{KZ}}]\label{l5.1}
Let $(X_n)$ be a  sequence in $\mathbb{L}^0_+(\bP)$ and let $X\in \mathbb{L}^0_+(\bP)$. Suppose that every FCC of $(X_n)$ converges to $X$ in probability.  Then every FCC of $(\abs{X_n-X})$ converges to $0$ in probability.
\end{lem}

\begin{proof}
Note that $(X_n-X)^-\leq X$ for all $n\in\N$ and $(X_n-X)^-\longrightarrow 0$ in probability. Let $\mathrm{d}\mu=\frac{1}{1+X}\mathrm{d}\bP$. Then $\mu$ is a finite measure on $(\Omega,\Sigma)$, $\mu\sim \bP$, and $X\in \OL(\mu) $. By Dominated Convergence Theorem, $(X_n-X)^-\longrightarrow 0$ in the $\OL(\mu)$-norm, and consequently, any FCC of $\big((X_n-X)^-\big)_n$  converges to $0$ in the $\OL(\mu)$-norm and thus also in the measure $\mu$. Note that the $\ZL(\mu)$- and $\ZL(\bP)$-topologies coincide. Hence, any FCC of $\big((X_n-X)^-\big)_n$  converges to $0$ in the probability $\bP$. The desired result now follows immediately from the equation
$\abs{X_n-X} = (X_n-X)+ 2(X_n-X)^-$.
\end{proof}

As remarked in Section \ref{intro},  the next two results also hold if  $\cK$ is assumed to be  a convex set in $\mathbb{L}^0_+(\bP)$ that is bounded in probability.
The {\em solid hull}  $\so (\cA)$ of a set $\cA$ is defined by $\so(\cA)=\big\{Y:\abs{Y}\leq \abs{X}\text{ for some }X\in\cA\big\}$.

\begin{prop}\label{p5.3}
Let $\cK$ be a convex bounded set in $\mathbb{L}^1_+(\bP)$.
Assume that the relative $\ZL(\bP)$-topology on $\cK$ is locally convex on a countable subset $\cS$ of $\cK$.
Then the relative $\ZL(\bP)$-topology on $\cK$ is uniformly locally convex-solid on  $\cS$.
\end{prop}

\begin{proof}
We first establish the special case where $\cS$ is a singleton set, say, $\cS=\{X\}$.
Again, we use the metric  given by $d(X',X'') = \E_\bP[\abs{X'-X''}\wedge \one]$ to generate the topology on $\ZL(\bP)$.  Let $\cU$ be a neighborhood of $0$ in $\ZL(\bP)$.
For each $n\in \N$, let  $\cB_n$ be the ball of radius $\frac{1}{n}$ centered at $0$ with respect to the metric $d$.
Set $$\cW_n = \co \so \big(\cB_n\cap (\cK-X)\big).$$
Then $\cW_n$  is a convex solid set (again, one may apply the Riesz Decomposition Theorem to verify solidity), and $(X+\cW_n)\cap \cK$ contains $(X+\cB_n) \cap \cK$.  Hence, $(X+\cW_n)\cap \cK$ is a neighborhood of $X$ in the relative $\ZL(\bP)$-topology on $\cK$.

It remains to show that $\cW_n\subseteq \cU$ for some $n\in \N$.
Assume the contrary.  Then we can find \emph{consecutive} finite subsets $I_n$ of $\N$ and  random variables
$\sum_{k\in I_n}a_kY_k \notin \cU$, where, for any $n\geq 1$,
$(Y_k)_{k\in I_n}\subseteq \so\big(\cB_n\cap (\cK-X)\big)$, $a_k \geq 0$ for $k\in I_n$, and $\sum_{k\in I_n}a_k = 1$.
Take random variables $(X_k)_{k\in I_n} \subseteq \cB_n \cap (\cK-X)$ such that $\abs{Y_k} \leq \abs{X_k}$ for each $k$.
Clearly, $X_k\longrightarrow 0$ in $\ZL(\bP)$, and hence $\cK\ni X+X_k\longrightarrow X$ in $\ZL(\bP)$.
By the local convexity of the relative $\ZL(\bP)$-topology on $\cK$ at $X$, every FCC of $(X+X_k)$ converges to $X$ in probability.
It then follows from  Lemma \ref{l5.1} that every FCC of $(\abs{X_k})$ converges to $0$ in probability.
In particular, $(\sum_{k\in I_n}a_k\abs{X_k})_n$ converges to $0$ in probability, and therefore so does $\big(\sum_{k\in I_n}a_kY_k\big)$, since $\bigabs{\sum_{k\in I_n}a_kY_k} \leq \sum_{k\in I_n}a_k\abs{X_k}$.
This contradicts  that $\sum_{k\in I_n}a_kY_k\notin \cU$ for all $n\in\N$ and thus proves the special case.

Now we consider the general case. Enumerate the set $\cS$ as a sequence $(Y_k)$.  For each $k$, by the special case and Theorem \ref{lcsolid}, there exists $\Q_k\sim \bP$ such that
if $(X_n)$ is a sequence in $\cK$ that converges in probability to  $Y_k$, then $(X_n)$ converges to $Y_k$ in $\OL(\Q_k)$.
Let $\ep >0$ be given.
Using the equivalence of \eqref{p2.11} and \eqref{p2.12} in Proposition \ref{p2.1}, for each $k$, there exists a measurable set $A_k$ with $\bP(A_k) > 1-\frac{\ep}{2^k}$
such that if $(X_n)$ is a sequence in $\cK$ that converges in probability to $Y_k$, then $\E\big[\abs{X_n-X}\one_{A_k}\big]\stackrel{n}{\longrightarrow} 0$.
Set $A = \cap_{k=1}^\infty A_k$.  Then $A$ is a measurable set with $\bP(A) > 1-\ep$.
By choice, if  $(X_n)$ is a sequence in $\cK$ that converges in probability to some $Y\in \cS$, then $\E_\bP\big[\abs{X_n-X}\one_{A}\big]\longrightarrow 0$.
By Proposition \ref{p2.1} and Theorem \ref{lcsolid} again, the relative $\ZL(\bP)$-topology on $\cK$ is uniformly locally convex-solid on  $\cS$.
\end{proof}

Let $\cK$ be a convex set in $\ZL(\bP)$. Say that a subset $\cS$ of $\cK$ is {\em relatively internal} in $\cK$ if for any $X\in \cK$, there exist $Y\in \cS$ and $t >0$ such that $Y + t(Y-X) \in \cK$, or equivalently, if for any $X\in \cK$, there exist $Z\in\cK$ and $0<\al<1$ such that $\al X+(1-\al) Z\in\cS$.
If $\cS$ is a singleton set, say, $\cS = \{Y\}$, then $\cS$ is relatively internal in $\cK$ if and only if
$0$
 is an internal point of the set $\spn (\cK-Y)$ in the vector space $\spn\{\cK -Y\}$ in the usual sense \cite[Definition V.1.6]{DS}.
The next result gives a sufficient condition on the set $\cK$ in order that (Q1+) has an affirmative answer.

\begin{thm}\label{t5.4}
Let $\cK$ be a convex bounded set in $\mathbb{L}_+^1(\bP)$ that contains a countable relatively internal subset $\cS$.  The following are equivalent.
\begin{enumerate}
\item\label{t5.41} The $\ZL(\bP)$-topology is locally convex on $\cK$.
\item\label{t5.42} There exists $\Q\sim \bP$ such that the $\ZL(\Q)$- and $\OL(\Q)$-topologies agree on $\cK$.
\end{enumerate}
\end{thm}

\begin{proof}
The implication \eqref{t5.42}$\implies$\eqref{t5.41} is trivial.  We show that \eqref{t5.41}$\implies$\eqref{t5.42}.   By Proposition \ref{p5.3} and  Theorem \ref{lcsolid},  choose $\Q\sim \bP$ such that if $(X_n)\subset \cK$ converges in probability to some $X\in \cS$, then $\E_\Q[\abs{X_n-X}]\longrightarrow 0$. Now let $(X_n)$ be a sequence in $\cK$ that converges in probability to some $X\in \cK$. We aim to show that $\E_\Q[\abs{X_n-X}]\longrightarrow 0$, which will complete the proof.
Choose $Z\in \cK$ and $\al \in(0,1)$ such that $\al X+(1-\al )Z\in\cS$.
Then $\big(\al X_n+ (1-\al) Z\big)$ is a sequence in $\cK$ that converges in probability to $\al X+ (1-\al) Z\in \cS$.
Thus
\[ \al \E_\Q[\abs{X_n-X}]= \E_\Q\Big[ \Bigabs{\big(\al X_n+ (1-\al) Z\big) -\big(\al X+ (1-\al) Z\big)}\Big] \longrightarrow 0.\]
Since $\alpha>0$, it follows that $\E_\Q[\abs{X_n-X}]\longrightarrow 0$, as desired.
\end{proof}

We are ready to prove Theorem \ref{t5}, which complements Theorem \ref{t1}.

 \begin{proof}[Proof of Theorem \ref{t5}]
In light of Theorem \ref{t5.4}, it suffices to show that $\cK = \co(X_n)_{n=1}^\infty$, where $(X_n)$ is a bounded positive sequence in $\OL(\bP)$, contains a countable subset $\cS$ that is relatively internal in $\cK$.
We claim that such a set is
\[ \cS = \Big\{\sum^m_{n=1}b_nX_n: m \in \N,\; \text{each } b_n\geq 0 \text{ is rational},\;\sum^m_{n=1}b_n = 1\Big\}.\]
Obviously, $\cS$ is a countable subset of $\cK$.
Suppose that $X = \sum^m_{n=1}a_nX_n \in \cK$, where  $a_n \geq 0$ for each $1\leq n\leq m$ and $\sum^m_{n=1}a_n = 1$.
Choose rational numbers $b_n \geq \frac{a_n}{3}$ for each $1\leq n\leq m$ such that $b: = \sum^m_{n=1}b_n \leq 1$.
Note that $b$ is a rational number.
Hence
\[Y= \sum^m_{n=1}b_nX_n + (1-b)X_{m+1} \in \cS.\]
By direct computation,
\[ Y+ \frac{1}{2}(Y-X) = \sum^m_{n=1}\big(\frac{3b_n}{2} - \frac{a_n}{2}\big)X_n + \frac{3}{2}(1-b)X_{m+1}.\]
By choice,  $\frac{3b_n}{2} - \frac{a_n}{2}\geq 0$ for $1\leq n\leq m$ and $\frac{3}{2}(1-b)\geq 0$.  Furthermore,
\[ \sum^m_{n=1}\big(\frac{3b_n}{2} - \frac{a_n}{2}\big) +  \frac{3}{2}(1-b) = \frac{3b}{2} - \frac{1}{2} +  \frac{3}{2}(1-b) = 1.\]
Thus $Y+ \frac{1}{2}(Y-X) \in \cK$.
This proves that $\cS$ is relatively internal in $\cK$.
\end{proof}

\section{Construction of Example A}

In this section, we give an example which shows that for a general convex bounded set $\cK$ in $\OL(\bP)$,  the $\ZL(\bP)$-topology on $\cK$ being uniformly locally convex-solid is strictly stronger than being locally convex. In fact, the set $\cK$ we construct is even $\ZL[0,1]$-compact and circled, i.e., $\cK=-\cK$. (Note that Theorem \ref{t2} holds for general solid sets, not necessarily positive, and that circledness is a reasonable weakening of solidity).

The example is a modification of an example of Pryce \cite{P}.
Denote the Lebesgue measure on $[0,1]$ by $\mathrm{m}$.
Let $(X_n)$ be a sequence of independent random variables in $\ZL[0,1]$, each of which obeys the Cauchy distribution with pdf $\frac{1}{\pi(1+t^2)}$, $t\in\R$.
Fix $1 < p <2$. For any $n\in\N$, let $$k_n = n\big(\log (n+2)\big)^p,$$
$$\beta_n = \log(1+k_n^2).$$
Define the  function  $F_n:\R \to \R $ by $$F_n(t) = \frac{t}{\beta_n}\one_{[-k_n,k_n]}(t),$$
and put
\begin{equation}\label{e4.00}Y_n = F_n(X_n).
\end{equation}
It is easily checked that
\[ \E_\m\big[\abs{Y_n}\big] = \int_\R \frac{|F_n(t)|}{\pi(1+t^2)}\mathrm{d}t = \frac{1}{\pi}\]
for all $n$.
Now, set
\begin{equation}\label{e4.0}
\cK = \Big\{\sum_{n=1}^\infty a_nY_n: \sum_{n=1}^\infty\abs{a_n} \leq 1\Big\}.
\end{equation}
It is clear that $\cK$ is a convex, circled, and bounded set in $\OL[0,1]$.

We now proceed to verify that $\cK$ satisfies the properties in Theorem \ref{cea}.

\begin{lem}\label{step 1}
Let $(a_i)_{i=1}^\infty$ be a sequence of real numbers and let $b_i = \frac{a_i}{\beta_i}$ for all $i\in\N$.
Fix $\ep >0$.  For any disjoint finite sets $I$ and $J$ in $\N$, let
\[P(I,J) =  \m\Big(\Bigabs{\sum_{i\in I}a_iY_i + \sum_{j\in J}b_jX_j} > \ep\Big).\]
If $I$ and $J$ are disjoint finite subsets of $\N$   and  $i_0\notin I\cup J$, then
\[ P\big(I\cup\{i_0\},J\big) \leq \frac{2}{\pi k_{i_0}}P(I,J) + P\big(I, J\cup\{i_0\}\big).\]
\end{lem}
The empty sum is conventionally regarded as $0$. In particular, $P(\emptyset,\emptyset)=\m(\emptyset)=0$.

\begin{proof}
We have
\begin{align}\label{eq3.1}
P\big(I\cup\{i_0\},J\big) =&
 \m\Big(\Big\{\Bigabs{\sum_{i\in I\cup\{i_0\}}a_iY_i + \sum_{j\in J}b_jX_j} > \ep\Big\}\cap \big\{\abs{X_{i_0}} > k_{i_0}\big\}\Big)\\ \notag
 & + \m\Big(\Big\{\Bigabs{\sum_{i\in I\cup\{i_0\}}a_iY_i + \sum_{j\in J}b_jX_j} > \ep\Big\}\cap \big\{\abs{X_{i_0}} \leq k_{i_0}\big\}\Big).
 \end{align}
Since $Y_{i_0} = 0$ on the set $\{\abs{X_{i_0}} > k_{i_0}\}$, the first term on the right is
\[
\m\Big(\Big\{\Bigabs{\sum_{i\in I}a_iY_i + \sum_{j\in J}b_jX_j} > \ep\Big\}\cap \big\{\abs{X_{i_0}} > k_{i_0}\big\}\Big)
= P(I,J) \cdot \m\big(|X_{i_0}| > k_{i_0}\big)\]
by independence.
Also,
\[ \m\big(\abs{X_{i_0}} > k_{i_0}\big) = \frac{2}{\pi}\int^\infty_{k_{i_0}}\frac{1}{1+t^2}\mathrm{d}t \leq \frac{2}{\pi k_{i_0}}.\]
Hence, the first term on the right of \eqref{eq3.1} is
$\leq \frac{2}{\pi k_{i_0}}P(I,J)$.
On the set $\{\abs{X_{i_0}}\leq  k_{i_0}\}$, $a_{i_0}Y_{i_0} = b_{i_0}X_{i_0}$.
Thus, the second term on the right in \eqref{eq3.1} is
\[\m\Big(\Big\{\Bigabs{\sum_{i\in I}a_iY_i + \sum_{j\in J\cup\{i_0\}}b_jX_j} > \ep\Big\}\cap \big\{\abs{X_{i_0}} \leq k_{i_0}\big\}\Big)\leq
P\big(I, J\cup\{i_0\}\big).\]
Combining the estimates above proves the lemma.
\end{proof}

It is well-known that if $J$ is a finite subset of $\N$ and $b_i, i\in J$, are real numbers, then
$\frac{1}{b}\sum_{j\in J}b_jX_j$
is Cauchy distributed, where $b = \sum_{j\in J}|b_j|$.  Hence, for $\ep > 0$,
\begin{equation}\label{eq3.2} \m\Big(\Bigabs{\sum_{j\in J}b_jX_j} > \ep\Big) = \frac{2}{\pi}\int^\infty_{\frac{\ep}{b}}\frac{1}{1+t^2}\mathrm{d}t
\leq \frac{2b}{\pi\ep}.
\end{equation}

\begin{lem}\label{est}
In the notation of Lemma \ref{step 1}, if $I$ and $J$ are disjoint finite subsets of $\N$, then
\[ P(I,J) \leq \frac{2}{\pi\ep}\prod_{i\in I}\Big(1 + \frac{2}{\pi k_i}\Big)\sum_{j\in J}\abs{b_j} + \frac{2}{\pi\ep}\sum_{i\in I}\Big[\abs{b_i}\prod_{i' \in I \bs\{i\}}\Big(1 + \frac{2}{\pi k_{i'}}\Big)\Big].\]
\end{lem}

The product over an empty index set is conventionally regarded as $1$.

\begin{proof}
The proof is by induction on the cardinality of $I$.
If $I = \emptyset$, then the result holds by (\ref{eq3.2}).
Suppose that the result holds for a set $I$ and let $i_0\notin I \cup J$.
For convenience, let us write
$A_M = \prod_{i\in M}(1 + \frac{2}{\pi k_i})$ for any finite subset $M$ of $\N$.
By Lemma \ref{step 1} and the inductive hypothesis, we have
\begin{align*}
&P\big(I\cup\{i_0\},J\big) \\
\leq &\frac{2}{\pi k_{i_0}}P(I,J) + P(I, J\cup\{i_0\})\\
\leq &\frac{2}{\pi k_{i_0}}\frac{2}{\pi\ep}\Big[A_I\sum_{j\in J}\abs{b_j} + \sum_{i\in I}A_{I \bs\{i\}}\abs{b_i}\Big]  + \frac{2}{\pi\ep}\Big[A_I\sum_{j\in J\cup\{i_0\}}\abs{b_j} + \sum_{i\in I}A_{I \bs\{i\}}\abs{b_i}\Big]\\
=&\frac{2}{\pi\ep}\Big(1 + \frac{2}{\pi k_{i_0}}\Big)A_I\sum_{j\in J}\abs{b_j} + \frac{2}{\pi\ep}A_I\abs{b_{i_0}} + \frac{2}{\pi\ep}\sum_{i\in I}\Big(1 + \frac{2}{\pi k_{i_0}}\Big)A_{I\bs\{i\}}\abs{b_i}\\
=& \frac{2}{\pi\ep}A_{I\cup\{i_0\}}\sum_{j\in J}\abs{b_j} + \frac{2}{\pi\ep}\sum_{i\in I\cup\{i_0\}}A_{(I\cup\{i_0\})\bs\{i\}}\abs{b_i}.
\end{align*}
This completes the induction.
\end{proof}

Taking $J = \emptyset$ in Lemma \ref{est} gives

\begin{lem}\label{maincor}
If $I$ is a finite set in $\N$ and $a_i, i\in I$, are real numbers, then, for any $\ep >0$,
\[ \m\Big(\Bigabs{\sum_{i\in I}a_iY_i} > \ep\Big) \leq \frac{2}{\pi\ep}\sum_{i\in I}\frac{\abs{a_i}}{\beta_i}\prod_{i'\in I\bs \{i\}}\Big(1+ \frac{2}{\pi k_{i'}}\Big).
\]
\end{lem}

\begin{prop}\label{fcc Y}
Any FCC of $(Y_1,-Y_1,Y_2,-Y_2,\dots)$ converges to $0$ in probability.
\end{prop}

\begin{proof}
Let $\ep > 0$, $I$ be a finite subset of $\N$ and $a_i, i\in I$, be real numbers such that $\sum_{i\in I}\abs{a_i} \leq 1$.
Observe that by the choice of $(k_n)$, $\sum_{n=1}^\infty \frac{1}{k_n} <\infty$ and hence $\prod^\infty_{i=1}(1+ \frac{2}{\pi k_{i}})$ converges to a nonzero finite number.
Therefore, there exists a finite constant $c$ such that $\prod_{i\in J}(1+ \frac{2}{\pi k_{i}})\leq c$ for any finite subset $J$ of $\N$. Let $i_0 = \min I$.
By Lemma \ref{maincor} and the fact that $(\beta_i)$ is an increasing positive sequence,
\[ \m\Big(\Bigabs{\sum_{i\in I}a_iY_i} > \ep\Big) \leq \frac{2c}{\pi\ep}\sum_{i\in I}\frac{\abs{a_i}}{\beta_i}\leq \frac{2c}{\pi\ep\beta_{i_0}}
.\]

Observe that if  $V \in \co(Y_j, -Y_j, Y_{j+1}, -Y_{j+1},\dots)$, then  there exists a finite set $I\subset \{j,j+1,\dots\}$  and real numbers $a_i$, $i\in I$, with  $\sum_{i\in I}\abs{a_i} \leq 1$ such that $V=\sum_{i\in I}a_iY_i$.
Thus, if $V \in \co(Y_j, -Y_j, Y_{j+1}, -Y_{j+1},\dots)$, then
\[ \m(\abs{V} > \ep) \leq  \frac{2c}{\pi\ep\beta_{j}}.\]
This, together with $\beta_j\longrightarrow\infty$, completes the proof of the lemma.
\end{proof}

We now proceed to a general result toward local convexity. Denote by $\cB_n$ the open ball of radius $\frac{1}{n}$ centered at $0$ with respect to  the metric $d(X',X'')=\E_\bP[\abs{X'-X''}\wedge \one]$ on $\ZL(\bP)$.

\begin{lem}\label{l4.6}
Let $(R_k)$ be a bounded sequence in $\OL(\bP)$ such that any FCC of $(R_1,-R_1,\linebreak R_2,-R_2,\dots)$ converges to $0$ in probability.  Set
\[\cL = \Big\{\sum_{k=1}^\infty a_kR_k: \sum_{k=1}^\infty\abs{a_k}\leq 1\Big\}.\]
Then for any $m\in\N$, there exists $n\in\N$ such that if $G \in \cB_n\cap \cL$, then $G = H+J$, where   $\E_\bP[\abs{H}]\leq \frac{1}{m}$ and $J\in \co(R_k, -R_k)^\infty_{k=m+1}$.
\end{lem}

\begin{proof}
Suppose otherwise.  We can find $m\in \N$ and a sequence $(G_n)$ with $G_n \in \cB_n\cap \cL$ such that  $G_n$ cannot be decomposed as desired for any $n\in\N$.  Write $G_n = \sum_{k=1}^\infty a_{nk}R_k$, where $\sum_{k=1}^\infty\abs{a_{nk}}\leq 1$ for each $n$.
By taking a subsequence if necessary, we may assume that $\lim_n a_{nk} = a_k$ exists for all $k\in\N$.
Note that $\sum_{k=1}^\infty \abs{a_k} \leq 1$.
Set $k_0=1$. Take $m_1$ such that $\sum_{k=1}^{k_0}\abs{a_{m_1,k}-a_k}\leq \frac{1}{2}$ and then take $k_1>k_0$ such that $\sum_{k=k_1+1}^\infty\abs{a_{m_1,k}}\leq \frac{1}{2^1}$. Now take $m_2>m_1$ such that $\sum_{k=1}^{k_1}\abs{a_{m_2,k}-a_k}\leq \frac{1}{2^2}$ and then take $k_2>k_1$ such that $\sum_{k=k_2+1}^\infty\abs{a_{m_2,k}}\leq \frac{1}{2^2}$. Repeating this process, we obtain a subsequence $(G_{m_n})$ of $(G_n)$ and a sequence $(k_n)$ in $\N$. Note that $G_{m_n}\in \cB_{m_n}\cap \cL\subset \cB_{n}\cap \cL$ for each $n\in\N$. Thus we abuse the notation to rewrite  $(G_{m_n})$ as $(G_n)$. Then
\begin{equation}\label{e4.3} \sum_{k=1}^{k_{n-1}}|a_{nk}-a_k| < \frac{1}{2^n}\quad\text{and}\quad \sum^\infty_{k=k_n+1}|a_{nk}| < \frac{1}{2^n}\end{equation}
for all $n\in\N$.
Let
\begin{equation}\label{e4.4}
U_n = \sum_{k=1}^{k_{n-1}}a_{nk}R_k, \quad V_n =\sum_{k=k_{n-1}+1}^{k_n}a_{nk}R_k \quad \text{and } W_n = \sum_{k=k_n+1}^\infty a_{nk}R_k. \end{equation}
Then $G_n = U_n +V_n +W_n$.  Clearly, $(U_n)$ converges to $\sum_{k=1}^\infty a_kR_k$ in $\OL(\bP)$ and hence in $\ZL(\bP)$, and $(W_n)$ converges to $0$
in $\OL(\bP)$ and hence in $\ZL(\bP)$.
Note that $(V_n)$ can be expressed as an FCC of $(R_1,-R_1,R_2,-R_2,\dots)$ and hence converges to $0$ in $\ZL(\bP)$ by assumption.
Since $(G_n)$ converges to $0$ in $\ZL(\bP)$,  $U_n = G_n -V_n - W_n\longrightarrow 0$ in $\ZL(\bP)$ as well.  Therefore, $\sum_{k=1}^\infty a_k R_k = 0$ a.s.
Let $H_n = U_n+W_n$ and $J_n = V_n$.
Then $G_n = H_n +J_n$, $(H_n)$ converges in $\OL(\bP)$ to $  0$, and $J_n \in \co(R_k,-R_k)_{k=k_{n-1}+1}^{k_n}$.  For sufficiently large $n$, we see that
\[ \E_\bP[\abs{H_n}]\leq \frac{1}{m}\;\;\text{ and   }\;\;J_n \in \co(R_k,-R_k)_{k=m+1}^{\infty},\] contrary to the choice of $G_n$'s.
This establishes the lemma.
\end{proof}

Recall that the convex-solid hull $\co\so(\cA)$ is convex and solid.  Furthermore, it is an easy fact that the solid hull of a convex set in $\mathbb{L}_+^0(\bP)$ is also convex.

\begin{prop}\label{p4.6}
Let $(R_k)$ be a bounded sequence in $\OL(\bP)$ and let
\[\cL = \Big\{\sum_{k=1}^\infty a_kR_k: \sum_{k=1}^\infty\abs{a_k}\leq 1\Big\}.\]
\begin{enumerate}
\item\label{p4.61} If every FCC of $(R_1,-R_1,R_2,-R_2,\dots)$ converges to $0$ in probability, then
 the $\ZL(\bP)$-topology on $\cL$ is locally convex at $0$.
\item\label{p4.62} If  every FCC of $(\abs{R_k})$ converges to $0$ in probability, then the $\ZL(\bP)$-topology on $\cL$ is locally convex-solid at $0$.
\end{enumerate}
\end{prop}

\begin{proof}
Let $r\in \N$ be given.  We will find $n\in\N$ such that $\co(\cB_n\cap \cL)\subseteq \cB_r$ in Case \eqref{p4.61} and $\co\so(\cB_n\cap \cL) \subseteq \cB_r$ in Case \eqref{p4.62}, from which the desired conclusions follow.
For Case \eqref{p4.62}, note that if $(U_k)$ is an FCC of $(R_1,-R_1,R_2,-R_2,\dots)$,
then there is an FCC $(V_k)$ of $(\abs{R_1},\abs{R_1},\abs{R_2},\abs{R_2},\dots)$ such that $\abs{U_k}\leq V_k$ for all $k$.  Hence in Case \eqref{p4.62}, every FCC of $(R_1,-R_1,R_2,-R_2,\dots)$
converges to $0$ in probability as well. Therefore, Lemma \ref{l4.6} applies in both cases.

Choose $s\in \N$ such that $\cB_s + \cB_s\subseteq \cB_r$. From the respective assumptions, there exists $m\in \N$ such that $\frac{1}{m}\cB_{\OL(\bP)} \subseteq \cB_s$ and that
\begin{align*}
\co(R_k,-R_k)^\infty_{k=m+1} & \subseteq \cB_s \text{ in Case (1)},\\
\co(|R_k|)^\infty_{k=m+1} & \subseteq \cB_s \text{ in Case (2)}.
\end{align*}
By  Lemma  \ref{l4.6}, there exists $n$ such that
\begin{equation}\label{e4.6}
\cB_n\cap \cL\subseteq  \frac{1}{m}\cB_{\OL(\bP)}+\co(R_k, -R_k)^\infty_{k=m+1}.\end{equation}
Since the right hand side of \eqref{e4.6} is a convex set, in Case \eqref{p4.61},
\[
\co(\cB_n\cap \cL) \subseteq  \frac{1}{m}
\cB_{\OL(\bP)}+\co(R_k, -R_k)^\infty_{k=m+1} \subseteq \cB_s + \cB_s \subseteq \cB_r.\]
In Case \eqref{p4.62},  note that $\co(R_k, -R_k)^\infty_{k=m+1} \subseteq \so\co(|R_k|)^\infty_{k=m+1}$ and the latter set is convex and solid.  It follows from \eqref{e4.6} that
\begin{equation}\label{e4.7} \cB_n\cap \cL  \subseteq  \frac{1}{m}\cB_{\OL(\bP)}+ \so\co(\abs{R_k})^\infty_{k=m+1}
\end{equation}
and that the right hand side is a convex solid set.
Therefore, since $\cB_s$ is solid,
\begin{align*}\co\so(\cB_n\cap \cL) &\subseteq \frac{1}{m}\cB_{\OL(\bP)}+ \so\co(\abs{R_k})^\infty_{k=m+1} \\&\subseteq \cB_s + \so \cB_s = \cB_s + \cB_s \subseteq \cB_r.\end{align*}
This completes the proof of the proposition.
\end{proof}

We need one more technical lemma toward local convexity of the $\ZL[0,1]$-topology on $\cK$.

\begin{lem}\label{l4.7}
Let $\cA$ be a convex circled set in a topological vector space $(\cX,\tau)$.
Then the relative $\tau$ topology on $\cA$ is locally convex if and only if it is locally convex at $0$.
\end{lem}

\begin{proof}
Let $x_0\in \cA$ and let $\cV$ be a $\tau$-neighborhood of $0$.
It suffices to produce a convex set $\cC$ and a $\tau$-neighborhood $\cU$ of $0$ such that
\[ (x_0+\cU)\cap \cA \subseteq \cC \subseteq (x_0+\cV)\cap \cA.\]
Since the relative $\tau$ topology on $\cA$ is locally convex at $0$, there is a  $\tau$-neighborhood $\cU$
of $0$ such that $\co\big(\frac{\cU}{2}\cap \cA\big) \subseteq \frac{\cV}{2}$.  Thus $\co(\cU\cap 2\cA) \subseteq \cV$.
Let
\[ \cC = \co\big((x_0+\cU)\cap \cA\big).\]  To complete the proof, we show that $\cC \subseteq (x_0+\cV)\cap \cA$.
Let $x\in \cC$.  Write $x = \sum^n_{i=1}a_i(x_0+x_i)$, where $a_i\geq 0$, $\sum^n_{i=1}a_i =1$, $x_i\in \cU$ and $x_0+x_i\in \cA$.
Then
\[ x_i = 2\big(\frac{x_0+x_i}{2}-\frac{x_0}{2}\big) \in 2\cA.\]
Hence $x_i\in \cU \cap 2\cA$.  Therefore, $\sum^n_{k=1}a_ix_i\in \co(\cU\cap 2\cA) \subseteq \cV$.
Thus $x = x_0 + \sum^n_{k=1}a_ix_i\in x_0+\cV$.  Clearly, $\cC\subseteq \cA$.
Hence, $x\in (x_0+\cV)\cap \cA$, as desired.
\end{proof}

Combining Proposition \ref{fcc Y},  Proposition \ref{p4.6} and Lemma \ref{l4.7}, we have
\begin{prop}\label{p4.8}
The $\ZL[0,1]$-topology on $\cK$ defined by \eqref{e4.0} is locally convex on $\cK$.
\end{prop}

The following results conclude $\ZL[0,1]$-compactness of $\cK$.

\begin{prop}\label{p4.9}
Let $\cB$ be the closed unit ball of $\ell^1$ with the relative $\sigma(\ell^1,c_0)$-topology (which coincides with the topology of coordinatewise convergence).
Suppose that $(R_k)$ is  a bounded sequence in $\OL(\bP)$ such that every FCC of $(R_1,-R_1,R_2,-R_2,\dots)$ converges to $0$ in probability.
Define a map $T: \cB \to \ZL[0,1]$ by $T\big((a_k)_k\big) = \sum_{k=1}^\infty a_k R_k$.  Then $T$ is continuous and $T(\cB)$ is compact in $\ZL(\bP)$.
\end{prop}

\begin{proof}
The second statement follows from the first one since  $\cB$ is $\sigma(\ell^1,c_0)$-compact. Note that the relative $\sigma(\ell^1,c_0)$-topology on $\cB$ is metrizable.
Let $(x_n)$ be a sequence in $\cB$ that converges coordinatewise to $x\in \cB$.  It is enough to show that a subsequence of $(Tx_n)$ converges to $Tx$ in $\ZL(\bP)$.
Write $x_n = (a_{nk})_k$ and $x = (a_k)_k$.
By passing to a subsequence, we may assume that the inequalities \eqref{e4.3} hold.
Define $U_n, V_n$ and $W_n$ as in \eqref{e4.4}.
Then $Tx_n = U_n +V_n +W_n$, $(U_n +W_n)$ converges in $\OL(\bP)$ to $\sum_{k=1}^\infty a_kR_k = Tx$, and $(V_n)$ converges to $0$ in probability by assumption.
It follows that $(Tx_n)$ converges to $Tx$ in probability, as desired.
\end{proof}

The following  is now immediate from Proposition \ref{fcc Y} and Proposition \ref{p4.9}.

\begin{cor}\label{c4.10}
The set $\cK$ defined by \eqref{e4.0} is a compact subset of $\ZL[0,1]$.
\end{cor}

We need a final fact to complete the proof of Theorem \ref{cea}.

\begin{lem}\label{Z_n}Let $(Y_n)$ be as defined in \eqref{e4.00}.
For each $n\in \N$, let
\[ Z_n = \frac{1}{n}\sum^n_{m=1}{|Y_{n+m}|}.\]
Then $(Z_n)$ converges to $\frac{\one}{\pi}$ in $L^2[0,1]$ and hence in probability.
\end{lem}

\begin{proof}
Since $\E_\m[\abs{Y_n}] = \frac{1}{\pi}$ for all $n$, $\E_\m[Z_n] = \frac{1}{\pi}$ for all $n\in\N$.
Also, $(\abs{Y_n})$ is a sequence of independent random variables.  Hence,
\begin{align*}
\E\Big[\bigabs{Z_n - \frac{1}{\pi}}^2\Big] =& \var(Z_n) = \frac{1}{n^2}\sum^n_{m=1}\var(\abs{Y_{m+n}})\\
 \leq& \frac{1}{n^2}\sum^n_{m=1}\E_\m[ Y_{m+n}^2]\\
 =& \frac{1}{n^2}\sum^n_{m=1}\frac{2}{\beta_{n+m}^2}\int^{k_{m+n}}_0
 \frac{t^2}{\pi(1+t^2)}\mathrm{d}t\\
 \leq& \frac{1}{n^2}\sum^n_{m=1}\frac{2k_{m+n}}{\pi\beta_{n+m}^2}
\leq \frac{2k_{2n}}{n\pi\beta_{n}^2}.
\end{align*}
It is easy to see that
$\beta_n \geq 2\log k_n \geq 2\log n$.
Therefore,
\[\E\Big[\bigabs{Z_n - \frac{1}{\pi}}^2\Big] \leq \frac{(\log (2n+2))^p}{\pi(\log n)^2} \longrightarrow 0
\]
as $n\longrightarrow \infty$, and the lemma is proved.
\end{proof}

\begin{proof}[Completion of proof of  Theorem \ref{cea}]  By Proposition \ref{p4.8} and Corollary \ref{c4.10}, it remains to verify that there does not exist $\Q \sim \m$ such that the $\ZL(\Q)$- and $\OL(\Q)$-topologies agree on $\cK$. Suppose otherwise. Let $\Q$ be as such.
Let $\cU$ be a $\ZL[0,1]$-neighborhood of $0$ such that $\frac{\one}{\pi}\notin \ol{\cU}$, where the closure is taken in $\ZL[0,1]$.
By Theorem \ref{lcsolid}, the $\ZL[0,1]$-topology on $\cK$ is locally convex-solid at $0$.
Hence there exists a convex solid set $\cW\subseteq \cU$ such that $\cW\cap \cK$ is a neighborhood of $0$ for the relative $\ZL[0,1]$-topology on $\cK$.
Note that $Y_n\longrightarrow 0$ in $\ZL[0,1] $ (see e.g.~Proposition  \ref{fcc Y}). Thus there exists $n_0\in\N$ such that $Y_n \in \cW$ for all $n> n_0$.
Then $Z_n \in \cW$ for all $n> n_0$.
But $Z_n \longrightarrow \frac{\one}{\pi}$ in probability by Lemma \ref{Z_n}.  Hence, $\frac{\one}{\pi} \in \ol{\cW} \subseteq \ol{\cU}$, contrary to the choice of $\cU$.
This contradiction completes the proof.
\end{proof}

We include a remark on the importance of positivity in Theorems \ref{t1} and \ref{t5}.

\begin{rem}
Put $\cK'=\co\big(\{0\}\cup (Y_n)_{n=1}^\infty\big)$. Then $\cK'\subset \cK$, so that the relative $\ZL[0,1]$-topology is also locally convex on $\cK'$. The same arguments as in the above proof show that there does not exist $\Q \sim \m$ such that the $\ZL(\Q)$- and $\OL(\Q)$-topologies agree on $\cK'$. Hence Theorem \ref{t5} fails without positivity. Since  $Y_n\longrightarrow 0$ in $\ZL[0,1] $ and $\ol{\cK'}\subset \cK$, the main implication  \eqref{t13}$\implies$\eqref{t14} in Theorem \ref{t1} fails without positivity as well.
\end{rem}

\section{Construction of Example B}

In this section, we construct a convex bounded set $\cK$ in $\mathbb{L}^1_+(\bP)$ on which the relative $\ZL(\bP)$-topology is locally convex but there does not exist $\Q\sim\bP$ such that the $\ZL(\Q)$- and $\OL(\Q)$-topologies coincide on $\cK$. This will establish our final result Theorem \ref{ceb}.

Let $(\Omega_0,\Sigma_0,\bP_0)$ be the two-point probability space on $\Omega_0=\{0,1\}$, where each point is given weight $\frac{1}{2}$.
Let $\Gamma$ be an \emph{uncountable} set and let $(\Om,\Sig,\bP)$ be the  product probability space of $\Gamma\times \N$-copies of  $(\Omega_0,\Sigma_0,\bP_0)$:
$$(\Om,\Sig,\bP)=\prod_{\Gamma\times \N}(\Omega_0,\Sigma_0,\bP_0);$$
cf.~\cite[p.91]{Lo}. Then $$\Om=\prod_{\Gamma\times \N}\{0,1\}=\{0,1\}^{\Gamma\times\N},$$ and a generic point in $\Omega$ is a function $\eta:\Gamma\times\N \to \{0,1\}$.
For a  subset $\Theta$ of $\Gamma$, let $$\Sigma_\Theta=\sigma\Big(\big\{\eta: \eta(\gamma,n) = 0\big\}:\; (\gamma,n)\in \Theta\times\N\Big).$$
Then $\Sigma_\Theta\subset \Sig$, and $\bP$ is nonatomic on $\Sigma_\Theta$. Furthermore, if $\Theta$ and $\Theta'$ are disjoint subsets of $\Gamma$, and $X$ and $Y$ are two random variables that are $\Sigma_\Theta$- and $\Sigma_{\Theta'}$-measurable, respectively, then $X$ and $Y$ are independent. Finally, note that if $A\in\Sig$ and $\bP(A)>0$, then by the construction of $\Sig$ and $\bP$, it is easy to see that there exist a subset $B$ of $A$ and a countable subset  $\Theta$ of $\Gamma$ such that $B\in\Sigma_\Theta$ and $\bP(A\backslash B)=0$.

Let $\gamma\in \Gamma$.
Define random variables on $\Om$ by
\begin{align*}
U_{\gamma,1} = &2\one_{\{\eta: \eta(\gamma,1) = 0\}} ,\\
U_{\gamma,n} = &U_{\gamma,1} + 2^n\one_{\{\eta: \eta(\gamma,i) = 0, 1\leq i\leq n\}},\;\; \text{ if }n\geq 2.
\end{align*}
Clearly, $U_{\gamma,n}\in \mathbb{L}^1_+(\Om,\Sig,\bP)$
and $\E_\bP[U_{\gamma,n}] \leq 2$ for any $(\gamma,n)\in \Gamma\times \N$.
If $\Theta \subseteq \Gamma$, let
\[ \cK_\Theta = \co\big\{U_{\gamma,n}: \gamma\in \Theta, n\in \N\big\},\]
and put
\begin{align}\label{ceb-K}
\cK=\cK_\Gamma=\co\big\{U_{\gamma,n}: \gamma\in \Gamma, n\in \N\big\}.
\end{align}
Clearly, every random variable in $\cK_\Theta$ is $\Sigma_\Theta$-measurable.
Note that if $X\in \cK$, then there is a finite set $\Theta\subseteq \Gamma$ such that $X\in \cK_\Theta$.
Moreover, for any set $\Theta\subseteq \Gamma$, note that $$\cK=\co\big(\cK_\Theta\cup \cK_{\Theta^c}\big)=\Big\{\al X+(1-\al) Y:\;0\leq \al\leq 1,X\in\cK_\Theta, Y\in \cK_{\Theta^c}\Big\}.$$

We first disprove existence of any $\Q$ with the required properties for the set $\cK$ constructed above.

\begin{prop}\label{nonequiv}
Let $\cK$ be as in \eqref{ceb-K}.
There does not exist any $\Q\sim \bP$ such that the $\ZL(\Q)$- and $\OL(\Q)$-topologies agree on $\cK$.
\end{prop}

\begin{proof}
If the  present proposition fails, then by Proposition \ref{p2.1}, there is a measurable set $A$ with $\bP(A) >0$ such that $\E_\bP\big[\abs{X_n-X}\one_A\big]\longrightarrow 0$ for any sequence $(X_n)$  in $\cK$ that converges in probability to some $X\in \cK$.
By replacing $A$ with a  subset having the same measure, we may assume that there exists a countable subset $\Theta$ of $\Gamma$ such that $A\in\Sigma_\Theta$.
Let $\gamma\in \Gamma\bs \Theta$.
Then $(U_{\gamma,n})_n$ is a sequence in $\cK$ that converges to $U_{\gamma,1}$ in probability.
Consider any $n\geq 2$.  Let $B_n = \big\{\eta: \eta(\gamma,i) = 0, 1\leq i\leq n\big\}$.
Since $A\in \Sigma_\Theta$ and $B_n \in \Sigma_{\{\gamma\}}$, $A$ and $B_n$ are independent sets.
Note that $U_{\gamma,n} - U_{\gamma,1} = 2^n$ on $B_n$ and $\bP(B_n)=\frac{1}{2^n}$.
Thus
\begin{align*}
\E_\bP\big[\abs{U_{\gamma,n}-U_{\gamma,1}}\one_A\big] \geq&\E_\bP\big[\abs{U_{\gamma,n}-U_{\gamma,1}}\one_{A\cap B_n}\big]
= 2^n\bP(A\cap B_n) = 2^n\bP(A)\bP(B_n) = \bP(A).
\end{align*}
This contradicts the choice of the set $A$ and concludes the proof.
\end{proof}

We now turn to the proof that the $\ZL(\bP)$-topology on $\cK$ is locally convex.

\begin{lem}\label{l5.7}
Let $X$ and $Y$ be random variables such that $\bP(\abs{X+Y}>\ep) < \ep$ for some $\ep>0$.
Assume that there exist a measurable set $A$ with $\bP(A) =\frac{1}{2}$ and a real number $c$ such that $X \leq c$ on $A$ and $X\geq c$ on $A^c$ and that $\one_A$ and $Y$ are independent.
Then $\bP(\abs{Y+c}> \ep) < 4\ep$.
\end{lem}

\begin{proof}
We have
\begin{align*}
\bP(A) \bP(Y < -c-\ep) = &\bP(A\cap\{Y < -c-\ep\}) \leq \bP(\{X\leq c\}\cap\{Y < -c-\ep\})\\
\leq &\bP(X+Y< -\ep)\leq \bP(\abs{X+Y} > \ep\} < \ep.
\end{align*}
Hence, $ \bP(Y < -c-\ep) < 2\ep$.
Similarly, by considering $A^c$, we obtain that $\bP(Y> -c+\ep) < 2\ep$.
Combining these two inequalities gives the desired result.
\end{proof}

\begin{lem}\label{p5.8}
Let $\Theta$ be a  finite subset of $\Gamma$ and $X\in \cK_\Theta$.
Suppose that $\big(\al_kX_k+(1-\al_k)Y_k\big)$ converges in probability to $X$,
where $X_k\in\cK_\Theta$, $Y_k\in \cK_{\Theta^c}$ and $0\leq \al_k\leq 1$ for all $k\in\N$.
Then $(\al_kX_k)$ converges to $X$ in probability and $\big((1-\al_k)Y_k\big)$ converges to $0$ in probability.
\end{lem}

\begin{proof}
There is a sequence $(\ep_k)$ decreasing to $0$ such that
\[ \bP\big(\abs{\al_k X_k -X + (1-\al_k)Y_k}>\ep_k\big) < \ep_k \text{ for all $k\in \N$.}
\]
Since $\al_kX_k-X$ is $\Sigma_\Theta$-measurable and $\bP$ is nonatomic on this $\sigma$-algebra, there exist a
set $A_k \in \Sigma_\Theta$ with $\bP(A_k) = \frac{1}{2}$ and a real number $c_k$ such that $\al_kX_k-X\leq c_k$ on $A_k$ and
$\al_kX_k-X\geq c_k$ on $A_k^c$.
By choice, $\one_{A_k}$ and $Y_k$ are independent.
Hence by Lemma \ref{l5.7}, $\bP\big(\abs{(1-\al_k)Y_k+c_k}> \ep_k\big) <4\ep_k$ for all $k$.
Therefore, $\big((1-\al_k)Y_k + c_k\big)$ converges to $0$ in probability.
It follows that $(\al_kX_k -c_k)$ converges to $X$ in probability.  To complete the proof, it suffices to show that $c_k\longrightarrow 0$.
Observe that since $X\in\cK_\Theta$ and $X_k\in \cK_\Theta$ for all $k$, all
$X_k$'s and $X $ vanish on the set
\[B = \cap_{\gamma\in \Theta}\big\{\eta: \eta(\gamma,1) = 1\big\}.\] Since  $\Theta$ is finite, $\bP(B) > 0$. Thus $ -c_k\one_B= (\al_kX_k-c_k)\one_B  \longrightarrow X\one_B = 0$ in $\ZL(\bP)$ implies that   $c_k \longrightarrow 0$, as desired.
\end{proof}

\begin{lem}\label{p5.9}
Let $\cK$ be as in \eqref{ceb-K}. Then no sequence in $\cK$ converges to $0$ in probability.
\end{lem}

\begin{proof}
Assume that some sequence $(X_k)$ in $\cK$ converges to $0$ in probability.
Choose a countable subset $\Theta$ of $\Gamma$ such that $X_k\in \cK_\Theta$ for all $k$.
Enumerate $\Theta$ as $\{\gamma_1,\gamma_2,\dots\}$ and express
\[ X_k = \sum^{m_k}_{j=1}a_{kj}V_{kj},
\]
where $m_k\in \N$, $a_{kj}\geq 0$, $\sum^{m_k}_{j=1}a_{kj} = 1$ and $V_{kj} \in \cK_{\{\gamma_j\}}$ if $1\leq j\leq m_k$.  For convenience, let $a_{kj} = 0$ if $j > m_k$.
Since $V_{kj} \in \cK_{\{\gamma_j\}}$, $V_{kj} \geq U_{\gamma_j,1}\geq 0$.
As a result, $(\sum^{m_k}_{j=1}a_{kj}U_{\gamma_j,1})_k$ converges to $0$ in probability.  In particular,
$(a_{kj})_k$ converges to $0$ for each $j$.
This allows us to perturb $\sum^{m_k}_{j=1}a_{kj}U_{\gamma_j,1}$ slightly, when $k$ is large, by removing the first few terms and adjusting coefficients of the remaining terms, ending up with a new convex combination. Thus by taking a subsequence of $k\in\N$ if necessary, we can find an FCC $(W_k)$ of $(U_{\gamma_j,1})_j$ such that
\[\E_\bP\Big[\Bigabs{\sum^{m_k}_{j=1}a_{kj}
U_{\gamma_j,1}-W_k}\Big]\longrightarrow 0\;\; \text{ as $k\to \infty$}.\]
In particular, $(W_k)$ converges to $0$ in probability.
Being bounded above by the constant $2$, $(U_{\gamma_j,1})_j$ is $\bP$-uniformly integrable, and thus so is $(W_k)$.
Therefore, $\E_\bP[W_k]\longrightarrow0$.
However, since $\E_\bP[U_{\gamma_j,1}]=1$ for all $j$, $E_\bP[W_k] =1$ for all $k$, a contradiction.
\end{proof}

We are ready to present the proof of the local convexity of the $\ZL(\bP)$-topology on $\cK$.

\begin{prop}\label{p5.10}
For any $X\in \cK$, there exists $\Q_X\sim\bP$ ($\Q_X$ depending on $X$) such that if $(X_k)$ is a sequence in $\cK$ that converges to $X$ in probability, then $(X_k)$ converges to $X$ in $\OL(\Q_X)$.
Consequently, the $\ZL(\bP)$-topology on $\cK$ is locally convex.
\end{prop}

\begin{proof}
The second statement is easily deduced from the first.
To prove the first statement, pick $X\in \cK$.
By Proposition \ref{p2.1}, it is enough to show that for any $\ep>0$, there is a measurable set $A$ with $\bP(A) > 1-\ep$ such that $\E_\bP[\abs{X_k-X}\one_A]\longrightarrow 0$ for any sequence $(X_k)$ in $\cK$ that converges to $X$ in probability. Let  $\ep>0$ be given. Choose a finite set $\Theta\subseteq \Gamma$ such that $X\in \cK_\Theta$.
Let $n\in \N$ be so large that $\frac{\#\Theta}{2^n} < \ep$.
Set
\begin{align*}
B =& \cup_{\gamma\in \Theta}\big\{\eta: \eta(\gamma,i) =0, 1\leq i \leq n\big\},\\
A = &B^c=\cap_{\gamma\in \Theta}\big\{\eta: \eta(\gamma,i) =1\text{ for some } 1\leq i \leq n\big\}.
\end{align*}
Then $\bP(B) < \ep$ and hence $\bP(A) > 1-\ep$.  For any $\gamma\in\Theta$ and $k\geq n$, $\one_{\{\eta:\eta(\gamma,i)=0, 1\leq i\leq k\}}=0$ on $A$, and thus $0\leq U_{\gamma,k}\one_{A}\leq 2+ 2^{n-1}$ for any $\gamma\in\Theta$ and $k\in\N$, so that $0\leq Y\one_{A}\leq 2+ 2^{n-1}$ for any $Y\in \cK_\Theta$.  Therefore, $\{Y\one_A: Y \in \cK_\Theta\}$ is $\bP$-uniformly integrable.

Let $(X_k)$ be a sequence in $\cK$ that converges to $X$ in probability.
Write
\[ X_k= \al_kY_k + (1-\al_k)Z_k, \text{ where $Y_k\in \cK_\Theta$, $Z_k \in \cK_{\Theta^c}$ and $0\leq \al_k \leq 1$ for all $k$}.\]
By Lemma \ref{p5.8}, $(\al_k Y_k)$ converges to $X$ in probability and $\big((1-\al_k)Z_k\big)$ converges to $0$ in probability.
By the above, $(Y_k\one_A)$ is $\bP$-uniformly integrable; hence, so is $(\al_kY_k\one_A)$. Thus, $$\E_\bP\big[\abs{\al_kY_k-X}\one_A\big]\longrightarrow 0.$$
If $(\al_k)$ does not converge to $1$, then, by considering a subsequence, we may assume that $(\al_k)$
converges to some $\al$ with $0\leq \al < 1$.
Then $Z_k=\frac{1}{1-\alpha_k}[(1-\alpha_k)Z_k]\longrightarrow\frac{1}{1-\alpha } \cdot0=0$ in probability, contrary to Lemma \ref{p5.9}.  Therefore, $(\al_k)$ converges to $1$.
Thus since $\E_\bP[Y]\leq 2$ for any $Y\in\cK$, $$\E_\bP\big[\abs{X_k-\alpha_kY_k}\big]=\E_\bP\big[\abs{1-\alpha_k}Z_k\big]
\longrightarrow 0.$$
It follows that $\E_\bP[\abs{X_k-X}\one_A]\longrightarrow 0$, as desired.
\end{proof}

Obviously, the set $\cK$ constructed is nonseparable in $\ZL(\bP)$.  Neither is $\cK$ closed in $\ZL(\bP)$. Indeed,   for any distinct sequence $(\gamma_j)$, $\big(2\one_{\{\eta:\eta(\gamma_j,1)=0\}}\big)$ is an independent identically distributed sequence with expectation $1$, and thus by Law of Large Numbers, $\one$ is the $\ZL(\bP)$-limit of the arithmetic averages of $(U_{\gamma_j,1})$, which all lie in $\cK$. But it is easy to see that $\one\not\in\cK$.
Thus the question (Q1') from \S1, which is a restricted version of (Q1+), remains open.

\appendix

\section{Alternative Proofs of Theorems \ref{t1} and \ref{t2}}
We close by proving Theorems \ref{t1} and \ref{t2} in the spirit of the present paper, which we believe   gives further insight
into said theorems.  We begin with one more lemma.  Once again, we will use the metric $d(X',X'') = \E_\bP[\abs{X'-X''}]$ to generate the $\ZL(\bP)$-topology.

\begin{lem}\label{l5.5}
Let $\cL$ be a convex circled set in $\OL(\bP)$.  Assume that the $\ZL(\bP)$-topology on $\cL$ is locally convex-solid at $0$.  Then it is uniformly locally convex  solid on $\cL$.
\end{lem}

\begin{proof}
Let $\cU$ be a $\ZL(\bP)$-neighborhood of $0$.  Then $\frac{\cU}{2}$ is also a $\ZL(\bP)$-neighborhood of $0$.
By assumption, there is a convex-solid set $\cW \subseteq \frac{\cU}{2}$ such that $\cW\cap \cL$ is a neighborhood of $0$ in the relative $\ZL(\bP)$-topology on $\cL$.
Thus there exists $r>0$ such that $\cB(0,r) \cap \cL \subseteq \cW\cap \cL$.
Let $X\in \cL$.  If $Y \in \cB(X,r)\cap \cL$, then $\frac{Y-X}{2} \in \cL$, since $\cL$ is convex and circled, and $\frac{Y-X}{2}\in \cB(0,r)$, since $\cB(0,r)$ is solid.
Hence, $ \frac{Y-X}{2} \in \cW\cap \cL\subseteq \cW$.
This shows that
\[ \cB(X,r)\cap \cL \subseteq (X+2\cW) \cap \cL.\]
Hence, $(X+2\cW)\cap \cL$ is  a neighborhood of $X$ in the relative $\ZL(\bP)$-topology on $\cL$.
Since $2\cW$ is a convex-solid set contained in $\cU$, the proof is complete.
\end{proof}

\begin{proof}[Proof of Theorem \ref{t1}]
The implications \eqref{t14}$\implies$\eqref{t13}$\implies$\eqref{t12}$\implies$\eqref{t11} are immediate.
Assume that \eqref{t11} holds. WLOG, assume that $(X_n) \cup \{X\}$ is bounded in $\mathbb{L}^1_+(\bP)$.
Let $R_n = X_n -X$ for any $n\in\N$ and
\[ \cL = \Big\{\sum_{k=1}^\infty a_kR_k: \sum_{k=1}^\infty \abs{a_k} \leq 1\}.\]
Note that $\cK \subseteq X + \cL$.
By Lemma \ref{l5.1}, every FCC of $(\abs{R_n})_n$ converges to $0$ in probability.
By Proposition \ref{p4.6}\eqref{p4.62}, the $\ZL(\bP)$-topology on $\cL$ is locally convex-solid at $0$.
By Lemma \ref{l5.5}, the $\ZL(\bP)$-topology on $\cL$ is uniformly locally convex-solid on $\cL$.
By Theorem \ref{t4}, there exists $\Q\sim \bP$ such that the $\ZL(\Q)$- and $\OL(\Q)$-topologies agree on $\cL$. If $(U_k)$ is an FCC of $(R_1,-R_1,R_2,-R_2,\dots)$,
then there is an FCC $(V_k)$ of $(\abs{R_1},\abs{R_1},\abs{R_2},\abs{R_2},\dots)$ such that $\abs{U_k}\leq V_k$ for all $k$. Hence  every FCC of $(R_1,-R_1,R_2,-R_2,\dots)$ also converges to $0$ in probability.
Therefore, it follows from Proposition \ref{p4.9} that $\cL$ is compact in $\ZL(\bP)$.
In particular, $\ol{\cK} \subseteq X+ \cL$.  Hence Condition \eqref{t14} of Theorem \ref{t1} holds. This proves \eqref{t11}$\implies$\eqref{t14}.
\end{proof}

\begin{proof}[Proof of Theorem \ref{t2}]
The implications \eqref{t25}$\implies$\eqref{t24}$\implies$\eqref{t23}$\implies$\eqref{t22} are immediate.  Assume that \eqref{t22} holds. Again, WLOG, assume that $\cK$ is bounded in $\mathbb{L}^1_+(\bP)$. By Proposition \ref{p5.3}, the $\ZL(\bP)$-topology on $\cK$ is locally convex-solid at $0$.  Apply Theorem \ref{lcsolid} with $\cS=\{0\}$ to conclude that there exists $\Q\sim \bP$ such that if $(X_n)$ is a sequence in $\cK$ that converges to $0$ in probability, then $(X_n)$ converges to $0$ in $\OL(\Q)$.

Let $\ep > 0$ be given. By Proposition \ref{p2.1},  there is a measurable set $A$ with $\bP(A) > 1-\ep$ such that $\E_\bP\big[\abs{X_n}\one_A\big]\longrightarrow 0$ for any sequence $(X_n)$ in $\cK$ that converges to $0$ in probability.
Let $(X_n)$ be a sequence in $\cK$ that is Cauchy in  probability.
We want to show that $\E_\bP[\abs{X_n-X_m}\one_A]\longrightarrow0$ as $n,m\longrightarrow\infty$, which implies \eqref{t25}  by Proposition \ref{p2.2}. Suppose otherwise. Then there exists $\delta>0$ and natural numbers $n_1<m_1<n_2<m_2<\cdots$ such that
\begin{align}\label{finaleq}
\E_\bP\big[\abs{X_{n_k}-X_{m_k}}\one_A]>\delta\quad\text{ for any }k\in\N.
\end{align}
On the other hand, clearly, $\big(X_{n_k}-X_{m_k}\big)_k$ converges to $0$ in probability, and hence so does the sequence $\big(\frac{\abs{X_{n_k}-X_{m_k}}}{2}\big)_k$.
Note that
\[ 0 \leq \frac{1}{2}\abs{X_{n_k}-X_{m_k}} \leq \frac{1}{2}(X_{n_k}+X_{m_k}) \in \cK.\]
Thus $\frac{\abs{X_{n_k}-X_{m_k}}}{2} \in \cK$, due to the positive solidity of $\cK$.
The choice of the set $A$ yields that $\E_\bP\big[\abs{X_{n_k}-X_{m_k}}\one_A]\longrightarrow0$, contradicting \eqref{finaleq}.
\end{proof}


\begin{thebibliography}{99}

\bibitem{AB:06}
C.~D.~Aliprantis, O.~Burkinshaw, \emph{Positive Operators}, Springer, the Netherlands, 2006.

\bibitem{AK:06}
F.~Albiac, N.~J.~Kalton, \emph{Topics in Banach Space Theory}, Graduate Texts in Mathematics 233, Springer, USA, 2006.

\bibitem{BS}
W.~Brannath, W.~Schachermayer, A bipolar theorem for $\mathbb{L}^0_+(\Om,\mathcal{F},\bP)$, \emph{S\'{e}minaire de Probabilit\'{e}s} XXXIII, Lecture Notes in Mathematics 1709, Springer, Berlin, 1999, 349-354.

\bibitem{DS}
N.~Dunford, J.T.~Schwartz, \emph{Linear Operators}, Part I, Wiley, New York, 1958.

\bibitem{GLX:18}
N.~Gao, D.~H.~Leung, F.~Xanthos, A local Hahn-Banach theorem and its applications, \emph{Archiv der Mathematik}, to appear.

\bibitem{K}
C.~Kardaras, Uniform integrability and local convexity in $\mathbb{L}^0$, \emph{Journal of Functional Analysis} 266, 2014, 1913-1927.

\bibitem{KZ}
C.~Kardaras, G.~\v{Z}itkovi\'{c}, Foward-convex convergence in probability of sequences of nonnegative random variables, {\em Proceedings of the American Mathematical Society} 141, 2013, 919-929.

\bibitem{Kom}
J.~Koml\'{o}s, A generalization of a problem of Steinhaus, {\em Acta Mathematica Hungarica} 18, 1967, 217-229.

\bibitem{Lo}
M.~Loeve, \emph{Probability Theory I}, 4th edition, Graduate Texts in Mathematics 45, Springer-Verlag, New York, 1977.

\bibitem{P}
J.~D.~Pryce, An unpleasant set in a non-locally-convex vector lattice, \emph{Proceedings of the Edinburgh Mathematical Society} 18, 1973, 229-233.

\bibitem{S}
H.~H.~Schaefer, \emph{Banach Lattices and Positive Operators}, Springer, New York, 1974.

\bibitem{R:91}
W.~Rudin, \emph{Functional Analysis}, 2nd edition, McGraw-Hill Inc., Singapore, 1991.

\end{thebibliography}
\end{document}